\RequirePackage{fix-cm}
\documentclass[11pt, reqno]{amsart}
\usepackage{amsmath}
\usepackage{setspace}
\usepackage{fancyhdr}
\usepackage{mathrsfs}
\usepackage{xifthen}
\usepackage{verbatim}
\usepackage{hyperref}
\usepackage[all]{xcolor}

\usepackage{float}

\usepackage{comment}

\usepackage{geometry}
\usepackage{fancyhdr}
\usepackage{textcomp}
\usepackage{amsthm}
\usepackage{amstext}
\usepackage{amsfonts}
\usepackage{amssymb}
\usepackage{fixltx2e}
\usepackage{graphicx}

\usepackage{tikz}
\usepackage{tikz-cd}
\usepackage[margin=1cm]{caption}
\usetikzlibrary{calc,backgrounds}
\usetikzlibrary{matrix,arrows,decorations.pathmorphing}

\usepackage{bigstrut}

\definecolor{mynicegreen}{RGB}{52,194,48}
\definecolor{mydarkteal}{RGB}{0,102,102}

\makeatletter

\newcommand{\schff}{X}
\newcommand{\adicff}{\mathcal{X}}
\newcommand{\Spa}{\mathrm{Spa}}

\newcommand{\HN}{\mathrm{HN}}

\newcommand{\rk}{\mathrm{rk}}
\newcommand{\slope}{\lambda}

\newcommand{\inj}{\hookrightarrow}
\newcommand{\surj}{\twoheadrightarrow}

\newcommand{\numslope}{m}

\newcommand{\strsheaf}[1][]{\Ocal_{#1}}

\newcommand{\completion}[1]{\widehat{#1}}

\newcommand{\genlift}{\widetilde}
\newcommand{\genred}{\overline}

\newcommand{\genslope}{\lambda}
\newcommand{\genperm}{\sigma}
\newcommand{\genpolygon}{\mathscr{P}}
\newcommand{\gensecpolygon}{\mathscr{Q}}
\newcommand{\genthirdpolygon}{\mathscr{R}}
\newcommand{\dualpolygon}[1]{{#1}^*}
\newcommand{\lineseg}[2]{\underline{#1}^{(#2)}}
\newcommand{\HNslope}[2]{\slope_{#2}(#1)}

\newcommand{\slodom}{\succeq}
\newcommand{\polyset}[1]{\mathbb{P}_{#1}}
\newcommand{\concrearr}[1]{\widehat{#1}}

\newcommand{\gensubvb}{\Dcal}
\newcommand{\genquotvb}{\Fcal}
\newcommand{\genextvb}{\Ecal}
\newcommand{\genvb}{\Ecal}
\newcommand{\gentorsor}{\Ecal}
\newcommand{\gensecvb}{\Fcal}

\newcommand{\genmodif}{\iota}
\newcommand{\modif}[1]{{#1}'}
\newcommand{\gentrivialization}{\beta}

\newcommand{\punctuation}[1]{{#1}^\circ}

\newcommand{\genisocrystal}{N}

\newcommand{\uniformizer}{\pi}
\newcommand{\pseudounif}{\varpi}
\newcommand{\genlocfield}{E}
\newcommand{\compmaxunram}[1]{\breve{#1}}

\newcommand{\FFresfield}{C}
\newcommand{\FFclosedpt}{\infty}
\newcommand{\genperfdring}{R}
\newcommand{\genpadicfield}{K}
\newcommand{\genperfalg}{A}
\newcommand{\integerring}[1]{\strsheaf[#1]}

\newcommand{\tilt}[1]{{#1}^\flat}

\newcommand{\powboundedring}[1]{{#1}^{\circ}}
\newcommand{\witt}{W}
\newcommand{\genfrob}{\phi}

\newcommand{\genscheme}{S}

\newcommand{\topspace}[1]{\vert #1 \vert}

\newcommand{\BdR}{B_{\dR}}
\newcommand{\Be}{B_e}
\newcommand{\gendRunif}{t}

\newcommand{\flagvar}{\mathscr{F}\ell}
\newcommand{\BBmap}[1]{\mathrm{BB}_{#1}}
\newcommand{\Newt}{\mathrm{Newt}}

\newcommand{\newtonmap}[2][]{\nu_{#1}(#2)}

\newcommand{\domcocharset}{X_*(T)^+}
\newcommand{\gencochar}{\mu}
\newcommand{\genmaxtorus}{T}
\newcommand{\genparabolic}{P}
\newcommand{\genborel}{B}
\newcommand{\genlevi}{M}

\numberwithin{equation}{section}

\usepackage{color}

\newcommand{\F}{\mathbb{F}}
\newcommand{\Q}{\mathbb{Q}}
\newcommand{\Z}{\mathbb{Z}}

\newcommand{\et}{\text{\'{e}t}}

\DeclareMathOperator{\GL}{GL}

\DeclareMathOperator{\Spec}{Spec\,}

\DeclareMathOperator{\Bun}{Bun}

\DeclareMathOperator{\Gr}{Gr}

\DeclareMathOperator{\Proj}{Proj\,}

\DeclareMathOperator{\dR}{dR}

\DeclareFontFamily{OT1}{rsfs}{}
\DeclareFontShape{OT1}{rsfs}{n}{it}{<-> rsfs10}{}
\DeclareMathAlphabet{\mathscr}{OT1}{rsfs}{n}{it}

\newcommand{\Dcal}{\mathcal{D}}
\newcommand{\Ecal}{\mathcal{E}}
\newcommand{\Fcal}{\mathcal{F}}

\newcommand{\Ocal}{\mathcal{O}}

\newcommand{\Ycal}{\mathcal{Y}}

\newtheorem{lemma}[subsubsection]{Lemma}
\newtheorem{prop}[subsubsection]{Proposition}
\newtheorem{cor}[subsubsection]{Corollary}

\theoremstyle{remark}
\newtheorem*{remark}{Remark}
\newtheorem{defn}[subsubsection]{Definition}
\newtheorem{example}[subsubsection]{Example}

\newtheorem*{thm*}{Theorem}
\makeatletter
\def\th@remark{%
  \thm@headfont{\bfseries}%
  \normalfont 
}
\def\imod#1{\allowbreak\mkern5mu({\operator@font mod}\,\,#1)}
\makeatother

\usepackage{enumitem}		
\theoremstyle{theorem}
\newtheorem{theorem}[subsubsection]{Theorem}

\numberwithin{equation}{section}

\makeatother

\begin{document}
	
	\tikzset{
		node style sp/.style={draw,circle,minimum size=\myunit},
		node style ge/.style={circle,minimum size=\myunit},
		arrow style mul/.style={draw,sloped,midway,fill=white},
		arrow style plus/.style={midway,sloped,fill=white},
	}
    
	\title{On nonemptiness of Newton strata in the $B_{\dR}^+$-Grassmannian for $\GL_n$}
   
    \author[S. Hong]{Serin Hong}
    \address{Department of Mathematics, University of Arizona, 617 N Santa Rita Ave, Tucson, AZ 85721}
    \email{serinh@math.arizona.edu}
    
    \begin{abstract} We study the Newton stratification in the $\BdR^+$-Grassmannian for $\GL_n$ associated to an arbitrary (possibly nonbasic) element of $B(\GL_n)$. Our main result classifies all nonempty Newton strata in an arbitrary minuscule Schubert cell. For a large class of elements in $B(\GL_n)$, our classification is given by some explicit conditions in terms of Newton polygons. For the proof, we proceed by induction on $n$ using a previous result of the author that classifies all extensions of two given vector bundles on the Fargues-Fontaine curve. 
    \end{abstract}
	
	\maketitle

	\tableofcontents
	
	\rhead{}

	\chead{}
\section{Introduction}

\subsection{Motivation and main result}$ $

The $\BdR^+$-Grassmannian is an analogue of the affine Grassmannian in $p$-adic geometry. It was introduced by Caraiani-Scholze \cite{CS_cohomunitaryshimura} to study the cohomology of certain Shimura varieties, and also used by Scholze-Weinstein \cite{SW_berkeley} as a crucial tool for the construction of local Shimura varieties. In addition, it played a fundamental role in the work of Fargues-Scholze \cite{FS_geomLL} on the geometrization of the local Langlands correspondence via the geometric Satake equivalence for $p$-adic groups.

The main objective of this paper is to study a natural stratification of the $\BdR^+$-Grassmannian known as the \emph{Newton stratification}, which we briefly describe now. Let us fix a connected reductive group $G$ over a finite extension $\genlocfield$ of $\Q_p$. We write $\Gr_G$ for the $\BdR^+$-Grassmannian for $G$, and $\Gr_{G, \gencochar}$ for the Schubert cell associated to a dominant cocharacter $\gencochar$ of $G$. For an complete algebraically closed extension $\FFresfield$ of $\genlocfield$, we have
\[ \Gr_G(\FFresfield) = G(\BdR)/G(\BdR^+) \quad \text{ and } \quad \Gr_{G, \gencochar}(\FFresfield) = G(\BdR^+)\gencochar(\gendRunif)^{-1}G(\BdR^+)/G(\BdR^+)\]
where $\BdR$ is the $p$-adic de Rham period ring with valuation ring $\BdR^+$, residue field $\FFresfield$ and a fixed uniformizer $\gendRunif$. The Cartan decomposition for $G$ induces a decomposition
\[ \Gr_G = \bigsqcup_{\gencochar \in \domcocharset} \Gr_{G, \gencochar}\]
where $\domcocharset$ denotes the set of all dominant cocharacters of $G$. Moreover, each Schubert cell $\Gr_{G, \gencochar}$ is related to the (diamond of the) $p$-adic flag variety $\flagvar(G, \gencochar)$ via a natural Bialynicki-Birula map
\[ \BBmap{\gencochar}: \Gr_{G, \gencochar} \longrightarrow \flagvar(G, \gencochar),\]
which is an isomorphism if $\gencochar$ is minuscule. 
In order to define the Newton stratification on $\Gr_G$ and its Schubert cells, we consider the stack $\Bun_G$ of $G$-bundles on the Fargues-Fontaine curve $\schff$. By the result of Fargues \cite{Fargues_Gbundle}, the topological space $\topspace{\Bun_G}$ of $\Bun_G$ is in natural bijection with the set $B(G)$ of Frobenius-conjugacy classes of elements of $G(\compmaxunram{\genlocfield})$, where $\compmaxunram{\genlocfield}$ as usual denotes the $p$-adic completion of the maximal unramified extension of $\genlocfield$. 
Given an element $b \in B(G)$, we write $\genvb_b$ for the corresponding $G$-bundle on $\schff$. 
The theorem of Beauville-Laszlo \cite{BL_beauvillelaszlothm} implies that a $G$-bundle on the Fargues-Fontaine curve is specified by the gluing data of the trivial $G$-bundles on $\Spec(\BdR^+)$ and $\schff - \FFclosedpt$, where $\FFclosedpt$ is a fixed closed point on $\schff$ with residue field $\FFresfield$ and completed local ring $\BdR^+$. If we fix $b \in B(G)$, for every point $x \in \Gr_G(\FFresfield)$ we can modify the gluing data for $\genvb_b$ by $x$ to obtain a new $G$-bundle $\genvb_{b, x}$. We thus obtain a map
\[ \Newt_b: \Gr_G(\FFresfield) \longrightarrow B(G)\]
which maps each $x \in \Gr_G(\FFresfield)$ to the element $\modif{b} \in B(G)$ corresponding to $\genvb_{b, x}$. For each Schubert cell $\Gr_{G, \gencochar}$, the Newton stratification associated to $b \in B(G)$ is a decomposition into subdiamonds
\[ \Gr_{G, \gencochar} = \bigsqcup_{\modif{b} \in B(G)} \Gr_{G, \gencochar, b}^{\modif{b}}\]
where $\Gr_{G, \gencochar, b}^{\modif{b}}(\FFresfield)$ is the preimage of $\modif{b}$ in $\Gr_{G, \gencochar}(\FFresfield)$ under the map $\Newt_b$.

The Newton stratification of minuscule Schubert cells was originally introduced in the aforementioned work of Caraiani-Scholze \cite{CS_cohomunitaryshimura} as a key tool for studying the fibers of the Hodge-Tate period map. It has also been used as a pivotal tool for studying the $p$-adic period domain by many authors, such as Chen-Fargues-Shen \cite{CFS_admlocus}, Shen \cite{Shen_HNstrata}, Chen \cite{Chen_FRconjnonbasic}, Viehmann \cite{Viehmann_weakadmlocNewton}, Nguyen-Viehmann \cite{NV_HNstrata}, and Chen-Tong \cite{CT_weakaddandnewton}. 

For the trivial element $b = 1$, a result of Rapoport \cite{Rapoport_basicminuscule} shows that the Newton stratum $\Gr_{G, \gencochar, b}^{\modif{b}}$ is nonempty if and only if $\modif{b}$ is an element of the set $B(G, -\gencochar)$ defined by Kottwitz \cite{Kottwitz_isocrystal}. When $b$ is basic, meaning that $\genvb_b$ is semistable, Chen-Fargues-Shen \cite{CFS_admlocus} and Viehmann \cite{Viehmann_weakadmlocNewton} extends the result of Rapoport to parametrize all nonempty Newton strata by a generalized Kottwitz set. However, for a general element $b \in B(G)$, no explicit parametrization is known for nonempty Newton strata in an arbitrary Schubert cell. 

In order to explain our main result, which  classifies all nonempty Newton strata in the Schubert cell $\Gr_{G, \gencochar}$ for $G = \GL_n$ and a minuscule cocharacter $\gencochar$, we need to set up some notations. 
Let us recall that, as observed by Kottwitz \cite{Kottwitz_isocrystal}, the set $B(\GL_n)$ is naturally identified with the set of concave polygons on the interval $[0, n]$ with rational slopes and integer breakpoints, where a polygon refers to a continuous piecewise linear function whose graph passes through the origin. Given an element $b \in B(\GL_n)$, we write $\newtonmap{b}$ for the corresponding polygon 
and often regard it as a tuple of rational numbers $(\newtonmap[1]{b}, \cdots, \newtonmap[n]{b})$ where $\newtonmap[i]{b}$ denotes the slope of $\newtonmap{b}$ on the interval $[i-1, i]$. 
We may also represent the dominant cocharacter $\gencochar$ of $\GL_n$ as an $n$-tuple of descending integers $(\gencochar_1, \cdots, \gencochar_n$) and regard it as a concave polygon on $[0, n]$ whose slope on $[i-1, i]$ is $\gencochar_i$. 

Given two arbitrary elements $b, ~\modif{b} \in B(\GL_n)$, our main result gives an inductive criterion for the nonemptiness of the Newton stratum $\Gr_{\GL_n, \gencochar, b}^{\modif{b}}$. 
Let us provide a brief description of the inductive criterion here and refer the readers to Theorem \ref{classification of nonempty minuscule Newton strata, general case} for a precise statement. 
If $b$ is basic, meaning that $\newtonmap{b}$ is a line segment, the desired classification is given by the aforementioned results of Chen-Fargues-Shen \cite{CFS_admlocus} and Viehmann \cite{Viehmann_weakadmlocNewton}. If $b$ is not basic, we have unique elements $a \in B(\GL_m)$ and $c \in B(\GL_{n-m})$ for some integer $m$ such that $\newtonmap{a}$ and $\newtonmap{c}$ together form a partition of $\newtonmap{b}$ with $\newtonmap{a}$ being the line segment in $\newtonmap{b}$ of maximum slope. The key observation for our main result is that $\Gr_{\GL_n, \gencochar, b}^{\modif{b}}$ is not empty if and only if there exist $\modif{a} \in B(\GL_m)$ and $\modif{c} \in B(\GL_{n-m})$ with the following properties:
\begin{enumerate}[label=(\roman*)]
\item\label{nonemptiness of smaller strata for inductive criterion} The Newton strata $\Gr_{\GL_m, \gencochar_1, a}^{\modif{a}}$ and $\Gr_{\GL_{n-m}, \gencochar_2, c}^{\modif{c}}$ are not empty for some minuscule cocharacters $\gencochar_1$ of $\GL_m$ and $\gencochar_2$ of $\GL_{n-m}$. 
\smallskip

\item\label{extension property for inductive criterion} The vector bundle $\genvb_{\modif{b}}$ arises as an extension of $\genvb_{\modif{c}}$ by $\genvb_{\modif{a}}$; in other words, there exists a short exact sequence of vector bundles
\[ 
\begin{tikzcd}
0 \arrow[r]& \genvb_{\modif{a}} \arrow[r] & \genvb_{\modif{b}} \arrow[r]  & \genvb_{\modif{c}} \arrow[r] & 0.
\end{tikzcd}
\]
\end{enumerate}
The cocharacters $\gencochar_1$ and $\gencochar_2$ in the property \ref{nonemptiness of smaller strata for inductive criterion} are uniquely determined by $\modif{a}$ and $\modif{c}$. Moreover, the property \ref{nonemptiness of smaller strata for inductive criterion} imposes explicit bounds on the slopes in $\newtonmap{\modif{a}}$ and $\newtonmap{\modif{c}}$, and consequently yields a finite list of candidates for $(\modif{a}, \modif{c})$. For each candidate, we can check the property \ref{extension property for inductive criterion} by a previous result of the author \cite{Hong_extvbfinal}. Then for each candidate with the property \ref{extension property for inductive criterion}, we can inductively proceed to check the property \ref{nonemptiness of smaller strata for inductive criterion}; indeed, if $\newtonmap{b}$ has $r$ distinct slopes, then $\newtonmap{c}$ has $r-1$ distinct slopes while $\newtonmap{a}$ is a line segment by construction. 

\begin{figure}[H]
\begin{tikzpicture}[scale=0.75]
		

		\coordinate (left) at (0, 0);
		\coordinate (q0) at (3, 3);
		\coordinate (q1) at (5, 4.2);
		\coordinate (q2) at (8, 4.8);
		

		\coordinate (p0) at (4.5, 1.4);
		\coordinate (p1) at (8, 1.7);

		\draw[step=1cm,thick, color=red] (left) -- (q0) --  (q1) -- (q2);
		\draw[step=1cm,thick, color=mynicegreen] (left) -- (p0) --  (p1);
		
		\draw [fill] (q0) circle [radius=0.05];		
		\draw [fill] (q1) circle [radius=0.05];		
		\draw [fill] (q2) circle [radius=0.05];		
		\draw [fill] (left) circle [radius=0.05];
		
		
		\draw [fill] (p0) circle [radius=0.05];		
		\draw [fill] (p1) circle [radius=0.05];		

		

		
		\path (q0) ++(-0.7, 0) node {\color{red}$\newtonmap{b}$};
		\path (p0) ++(-0.5, -0.7) node {\color{mynicegreen}$\newtonmap{\modif{b}}$};

\end{tikzpicture}
\begin{tikzpicture}[scale=0.5]
        \pgfmathsetmacro{\textycoordinate}{4}
		\draw[->, line width=0.6pt] (0, \textycoordinate) -- (1.5,\textycoordinate);
		\draw (0,0) circle [radius=0.00];	
\end{tikzpicture}
\begin{tikzpicture}[scale=0.75]
		

		\coordinate (left) at (0, 0);
		\coordinate (q0) at (3, 3);
		\coordinate (q1) at (5, 4.2);
		\coordinate (q2) at (8, 4.8);
		
		\coordinate (r1) at (2.5, 1.2);
		\coordinate (r2) at (5, 1.7);

		\coordinate (p0) at (4.5, 1.4);
		\coordinate (p1) at (8, 1.7);
				
		\draw[step=1cm,thick, color=orange] (left) -- (q0);
		\draw[step=1cm,thick, color=violet] (q0) -- (q1) -- (q2);
		\draw[step=1cm,thick, color=mynicegreen] (left) -- (p0) --  (p1);
		\draw[step=1cm,thick, color=blue] (left) -- (r1) -- (r2);
		\draw[step=1cm,thick, color=magenta] (r2) -- (p1);
		
		\draw [fill] (q0) circle [radius=0.05];		
		\draw [fill] (q1) circle [radius=0.05];		
		\draw [fill] (q2) circle [radius=0.05];		
		\draw [fill] (left) circle [radius=0.05];
		
		\draw [fill] (r1) circle [radius=0.05];		
		\draw [fill] (r2) circle [radius=0.05];	
		
		\draw [fill] (p0) circle [radius=0.05];		
		\draw [fill] (p1) circle [radius=0.05];		

		

		
		\path (q0) ++(-1.5, -0.7) node {\color{orange}$\newtonmap{a}$};
		\path (q0) ++(1.3, 1.4) node {\color{violet}$\newtonmap{c}$};
		\path (r1) ++(0.7, 0.6) node {\color{blue}$\newtonmap{\modif{c}}$};
		\path (p0) ++(-0.5, -0.7) node {\color{mynicegreen}$\newtonmap{\modif{b}}$};
		\path (r2) ++(1.4, 0.4) node {\color{magenta}$\newtonmap{\modif{a}}$};

\end{tikzpicture}
\caption{Illustration of the inductive criterion}
\end{figure}
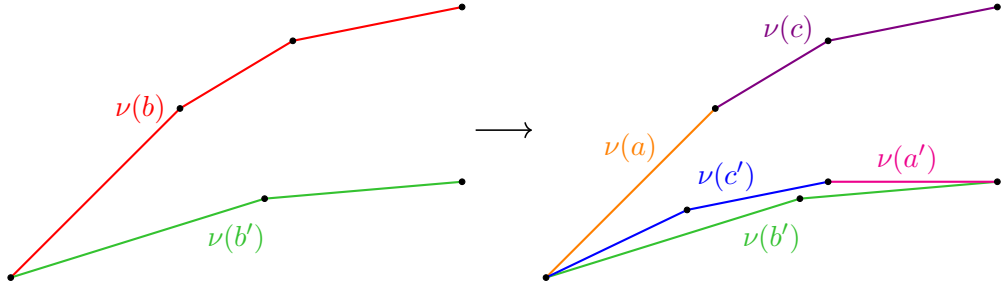

For a concrete example, we illustrate how our inductive criterion shows the nonemptiness of the stratum $\Gr_{\GL_8, \gencochar, b}^{\modif{b}}$ with
\[ 
\begin{aligned}
\newtonmap{b} &= \left(\dfrac{2}{3},\dfrac{2}{3}, \dfrac{2}{3}, \dfrac{3}{5}, \dfrac{3}{5}, \dfrac{3}{5}, \dfrac{3}{5}, \dfrac{3}{5} \right),\\
\newtonmap{\modif{b}} &= \left(\dfrac{1}{4}, \dfrac{1}{4}, \dfrac{1}{4}, \dfrac{1}{4}, 0, 0, 0, 0\right),\\
\gencochar &= (1, 1, 1, 1, 0, 0, 0, 0). 
\end{aligned}
\]
The elements $a \in B(\GL_3)$ and $c \in B(\GL_{5})$ are given by
\[ 
\newtonmap{a} = \left(\dfrac{2}{3}, \dfrac{2}{3}, \dfrac{2}{3}\right) \quad \text{ and } \quad \newtonmap{c} = \left(\dfrac{3}{5}, \dfrac{3}{5}, \dfrac{3}{5}, \dfrac{3}{5}, \dfrac{3}{5}\right). 
\]
We apply the inductive criterion with $\modif{a} \in B(\GL_3)$ and $\modif{c} \in B(\GL_5)$ given by
\[ 
\newtonmap{\modif{a}} = \left(-\dfrac{1}{3}, -\dfrac{1}{3}, -\dfrac{1}{3}\right) \quad \text{ and } \quad \newtonmap{\modif{c}} = \left(\dfrac{1}{2}, \dfrac{1}{2}, \dfrac{1}{2}, \dfrac{1}{2}, 0\right). 
\]
Indeed, the nonemptiness of the stratum $\Gr_{\GL_8, \gencochar, b}^{\modif{b}}$ follows from the following statements:
\begin{itemize}
\item $\genvb_{\modif{b}}$ arises as an extension of $\genvb_{\modif{c}}$ by $\genvb_{\modif{a}}$.
\smallskip

\item $\Gr_{\GL_3, \gencochar_1, a}^{\modif{a}}$ with $\gencochar_1 = (1, 1, 1)$ and $\Gr_{\GL_{5}, \gencochar_2, c}^{\modif{c}}$ with $\gencochar_2 = (1, 0, 0, 0, 0)$ are not empty. 
%
\end{itemize}
For the second statement, we note that $a$ and $c$ are basic for $\newtonmap{a}$ and $\newtonmap{c}$ being line segments. 

A special case of our main result reduces to a noninductive criterion as follows: 
\begin{theorem}\label{classification of nonempty Newton strata, case of sufficiently distinct slopes intro}
Let $\gencochar$ be a minuscule dominant cocharacter of $G = \GL_n$ represented by an $n$-tuple with entries $0$ and $1$. Given two arbitrary elements $b, \modif{b} \in B(\GL_n)$ such that 
the difference between any two distinct slopes in $\newtonmap{b}$ is greater than $1$, 
the Newton stratum $\Gr_{\GL_n, \gencochar, b}^{\modif{b}}$ is nonempty if and only if the following conditions are satisfied:
\begin{enumerate}[label=(\roman*)]
\item\label{modified Kottwitz set condition for nonempty Newton strata, intro} The polygon $\newtonmap{\modif{b}}$ lies below the polygon $\newtonmap{b} + \dualpolygon{\gencochar}$ with the same endpoints, where $\dualpolygon{\gencochar}$ denotes the unique dominant cocharacter of $\GL_n$ in the conjugacy class of $\gencochar^{-1}$. 
\smallskip

\item\label{slopewise dominance condition for nonemtpy Newton strata, intro} 
We have inequalities
\[ \newtonmap[i]{\modif{b}} \leq \newtonmap[i]{b} \leq \newtonmap[i]{\modif{b}} +1 \quad \text{ for } i = 1, \cdots, n.\]

\item\label{breakpoint condition for nonempty Newton strata, intro} For each breakpoint of $\newtonmap{b}$, there exists 
a breakpoint of $\newtonmap{\modif{b}}$ with the same $x$-coordinate. 
\end{enumerate}
\end{theorem}
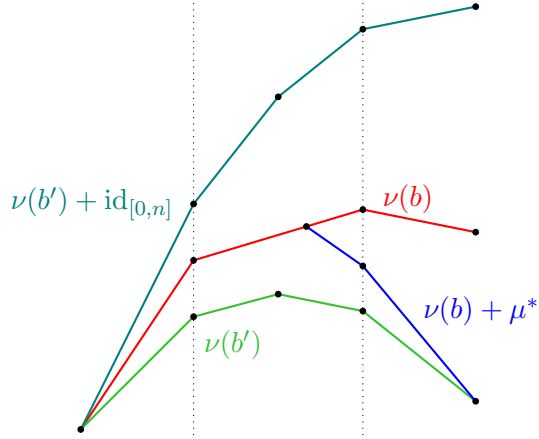
\begin{figure}[H]
\begin{tikzpicture}[scale=0.75]
		\coordinate (p00) at (2, 4);
		\coordinate (p11) at (3.5, 5.9);
		\coordinate (p22) at (5, 7.1);
		\coordinate (p33) at (7, 7.5);

		\coordinate (left) at (0, 0);
		\coordinate (q0) at (2, 3);
		\coordinate (q1) at (5, 3.9);
		\coordinate (q2) at (7, 3.5);
		
		\coordinate (r1) at (4, 3.6);
		\coordinate (r2) at (5, 2.9);

		\coordinate (p0) at (2, 2);
		\coordinate (p1) at (3.5, 2.4);
		\coordinate (p2) at (5, 2.1);
		\coordinate (p3) at (7, 0.5);
				
		\draw[step=1cm,thick, color=red] (left) -- (q0) --  (q1) -- (q2);
		\draw[step=1cm,thick, color=mynicegreen] (left) -- (p0) --  (p1) -- (p2) -- (p3);
		\draw[step=1cm,thick, color=teal] (left) -- (p00) --  (p11) -- (p22) -- (p33);
		\draw[step=1cm,thick, color=blue] (r1) -- (r2) -- (p3);
		
		\draw [fill] (q0) circle [radius=0.05];		
		\draw [fill] (q1) circle [radius=0.05];		
		\draw [fill] (q2) circle [radius=0.05];		
		\draw [fill] (left) circle [radius=0.05];
		
		\draw [fill] (r1) circle [radius=0.05];		
		\draw [fill] (r2) circle [radius=0.05];	
		
		\draw [fill] (p0) circle [radius=0.05];		
		\draw [fill] (p1) circle [radius=0.05];		
		\draw [fill] (p2) circle [radius=0.05];		
		\draw [fill] (p3) circle [radius=0.05];		

		\draw [fill] (p00) circle [radius=0.05];		
		\draw [fill] (p11) circle [radius=0.05];		
		\draw [fill] (p22) circle [radius=0.05];	
		\draw [fill] (p33) circle [radius=0.05];	
		
		\draw[step=1cm,dotted] (2, -0.1) -- (2, 7.6);
		\draw[step=1cm,dotted] (5, -0.1) -- (5, 7.6);

		
		\path (q1) ++(0.8, 0.2) node {\color{red}$\newtonmap{b}$};
		\path (r2) ++(2.1, -0.8) node {\color{blue}$\newtonmap{b} + \dualpolygon{\gencochar}$};
		\path (p00) ++(-1.8, 0) node {\color{teal}$\newtonmap{\modif{b}} + \mathrm{id}_{[0, n]}$};
		\path (p0) ++(0.7, -0.5) node {\color{mynicegreen}$\newtonmap{\modif{b}}$};

\end{tikzpicture}
\caption{Illustration of the conditions in Theorem \ref{classification of nonempty Newton strata, case of sufficiently distinct slopes intro}}
\end{figure}

The condition \ref{modified Kottwitz set condition for nonempty Newton strata, intro} is in fact equivalent to having $\modif{b}$ in the generalized Kottwitz set considered by Chen-Fargues-Shen \cite{CFS_admlocus} and Viehmann \cite{Viehmann_weakadmlocNewton}. When $b$ is basic, the condition \ref{modified Kottwitz set condition for nonempty Newton strata, intro} also implies the conditions \ref{slopewise dominance condition for nonemtpy Newton strata, intro} and \ref{breakpoint condition for nonempty Newton strata, intro}. Hence when $b$ is basic Theorem \ref{classification of nonempty Newton strata, case of sufficiently distinct slopes intro} agrees with the aforementioned result of Chen-Fargues-Shen \cite{CFS_admlocus} and Viehmann \cite{Viehmann_weakadmlocNewton}. 

The hypothesis on the cocharacter $\gencochar$ having entries $0$ and $1$ is insignificant; indeed, without this assumption we still get a similar statement by a simple reduction technique as stated in Proposition \ref{reduction to minimal minuscule}. On the other hand, the hypothesis on the slopes in $\newtonmap{b}$ is crucial. For the general case, the conditions \ref{modified Kottwitz set condition for nonempty Newton strata, intro} and \ref{slopewise dominance condition for nonemtpy Newton strata, intro} are still necessary but not sufficient.

\subsection{Outline of the proof}$ $

Given a vector bundle $\genvb$ on the Fargues-Fontaine curve $\schff$, its \emph{minuscule effective modification at $\FFclosedpt$} of degree $d$ refers to an injective bundle map $\modif{\genvb} \inj \genvb$ whose cokernel is the skyscraper sheaf at $\FFclosedpt$ with value $\FFresfield^{\oplus d}$. 
The Newton stratum $\Gr_{\GL_n, \gencochar, b}^{\modif{b}}$ is not empty if and only if there exists a minuscule effective modification $\genvb_{\modif{b}} \inj \genvb_b$ at $\FFclosedpt$.
We thus wish to classify all minuscule effective modifications of $\genvb_b$ at $\FFclosedpt$. If $b$ is basic, the desired classification is given by the aforementioned results of Chen-Fargues-Shen \cite{CFS_admlocus} and Viehmann \cite{Viehmann_weakadmlocNewton}. Let us now assume that $b$ is not basic. 
We can find a direct sum decomposition
\[ \genvb_b \simeq \genvb_a \oplus \genvb_c \quad \text{ with } \quad a \in B(\GL_m) \text{ and } c \in B(\GL_{n-m})\]
where $a$ is basic such that $\newtonmap{a}$ equals the line segment in $\newtonmap{b}$ of maximum slope. For every minuscule effective modification $\genmodif: \genvb_{\modif{b}} \inj \genvb_b$ at $\FFclosedpt$, the above decomposition extends to a commutative diagram of short exact sequences
\begin{equation*}\label{extension diagram of minuscule modifications}
\begin{tikzcd}
0 \arrow[r]& \genvb_a \arrow[r] & \genvb_b \arrow[r]& \genvb_c \arrow[r]& 0\\
0 \arrow[r]& \genvb_{\modif{a}} \arrow[r] \arrow[u, "\alpha", hookrightarrow] & \genvb_{\modif{b}} \arrow[r] \arrow[u, "\beta", hookrightarrow] & \genvb_{\modif{c}} \arrow[r] \arrow[u, "\gamma", hookrightarrow]& 0
\end{tikzcd}
\end{equation*}
where $\alpha$ and $\gamma$ are also minuscule effective modifications at $\FFclosedpt$. Conversely, given such a commutative diagram we apply a result of Chen-Tong \cite{CT_weakaddandnewton} to observe that $\alpha$ and $\gamma$ can be adjusted so that $\beta$ is a minuscule effective modification at $\FFclosedpt$. Then we use a previous result of the author \cite{Hong_extvbfinal} to classify all vector bundles $\genvb_{\modif{a}}$ and $\genvb_{\modif{c}}$ that fit into such a commutative diagram, and consequently proceed by induction to obtain the desired classification.

\subsection{Notations and conventions}$ $

Throughout the paper, we fix the following data:
\begin{itemize}
\item $\genlocfield$ is a finite extension of $\Q_p$. 
\smallskip

\item $\FFresfield$ is a complete and algebraically closed extension of $\genlocfield$. 
\smallskip

\item $G$ is a reductive group over $\genlocfield$ with Borel subgroup $\genborel$ and maximal torus $\genmaxtorus \subseteq \genborel$. 
\end{itemize}
We also retain the following notations:
\begin{itemize}
\item $\compmaxunram{\genlocfield}$ is the $p$-adic completion of the maximal unramified extension of $\genlocfield$.  
\smallskip


\item $B(G)$ is the set of Frobenius-conjugacy classes of elements of $G(\compmaxunram{\genlocfield})$. 
\smallskip

\item $\domcocharset$ is the set of all dominant cocharacters of $G$.
\end{itemize}
In addition, we use the following standard notations:
\begin{itemize}
\item Given a valued field $\genpadicfield$, we write $\integerring{\genpadicfield}$ for its valuation ring. 
\smallskip

\item Given a ringed space $\genscheme$, we write $\strsheaf[\genscheme]$ for its structure sheaf.
\smallskip

\smallskip

\item Given a perfectoid ring $\genperfdring$, we write $\tilt{\genperfdring}$ for its tilt and $\powboundedring{\genperfdring}$ for its subring of power bounded elements. 
\smallskip

\item Given a perfect $\F_p$-algebra $\genperfalg$, we write $\witt(\genperfalg)$ for the ring of Witt vectors over $\genperfalg$. 
\end{itemize}

\subsection*{Acknowledgments} 
The author is grateful to Miaofen Chen, Jilong Tong, and Eva Viehmann for their valuable comments which led to the correct formulation of the main result. The author is also thankful to the anonymous referee for helpful suggestions.  
This work was partially supported by the Simons Foundation under Grant Number 814268 while the author was a Simons Postdoctoral Fellow at the Simons Laufer Mathematical Sciences Institute in Berkeley, California. 



\section{Preliminaries}\label{background}

In this section, we review some basic facts about the $\BdR^+$-Grassmannian and $G$-bundles on the Fargues-Fontaine curve.


\subsection{The $\BdR^+$-Grassmannian}$ $

\begin{prop}[{\cite[Proposition 2.4]{Fontaine_BdR}}, {\cite[Lemma 3.6.3]{KL_relpadic1}}]\label{relative fontaine map}
Let $\genperfdring$ be a perfectoid algebra over $\FFresfield$. There exists a natural surjective homomorphism $\witt(\tilt{\powboundedring{\genperfdring}}) \surj \powboundedring{\genperfdring}$ whose kernel is a principal ideal of $\witt(\tilt{\powboundedring{\genperfdring}})$. 
\end{prop}

\begin{defn}
Let $\genperfdring$ be a perfectoid algebra over $\FFresfield$. Choose a generator $\gendRunif$ of the kernel of the map $\witt(\tilt{\powboundedring{\genperfdring}}) \surj \powboundedring{\genperfdring}$ in Proposition \ref{relative fontaine map}. We write $\BdR^+(\genperfdring)$ for the $\gendRunif$-adic completion of $\witt(\tilt{\powboundedring{\genperfdring}})[1/p]$, and define the \emph{de Rham period ring} associated to $\genperfdring$ by $\BdR(\genperfdring):= \BdR^+(\genperfdring)[1/\gendRunif]$.
\end{defn}

\begin{prop}[{\cite[Proposition 2.17]{Fontaine_BdR}}]\label{BdR as a discrete valued field} The ring $\BdR(\FFresfield)$ is a discretely valued field with valuation ring $\BdR^+(\FFresfield)$ and residue field $\FFresfield$. 
\end{prop}


We will henceforth write $\BdR:= \BdR(\FFresfield)$ and $\BdR^+:= \BdR^+(\FFresfield)$. We also fix a uniformizer $\gendRunif$ of $\BdR$ 
in light of Proposition \ref{BdR as a discrete valued field}. 

\begin{defn}
The \emph{$\BdR^+$-Grassmannian} is the functor $\Gr_G$ that associates to each perfectoid affinoid algebra $(\genperfdring, \genperfdring^+)$ over $\FFresfield$ the set of pairs $(\gentorsor, \gentrivialization)$ consisting of a $G$-torsor $\gentorsor$ over $\Spec(\BdR^+(\genperfdring))$ and a trivialization $\gentrivialization$ of $\gentorsor$ over $\Spec(\BdR(\genperfdring))$. 
\end{defn}

\begin{prop}[{\cite[Proposition 19.1.2]{SW_berkeley}}]\label{points of BdR-Grassmannian}
There exists a natural identification 
\[\Gr_G(\FFresfield) \cong G(\BdR)/G(\BdR^+).\]
\end{prop}

\begin{remark}
In fact, we can naturally identify $\Gr_G$ as the \'etale sheafification of the functor that associates to each perfectoid affinoid algebra $(\genperfdring, \genperfdring^+)$ over $\FFresfield$ the coset $G(\BdR(\genperfdring))/G(\BdR^+(\genperfdring))$. 
\end{remark}

\begin{prop}[{\cite[Corollary 19.3.4]{SW_berkeley}}]\label{schubert cells as diamonds}
Given $\gencochar \in \domcocharset$, there exists a locally spatial diamond $\Gr_{G, \gencochar}$ with
\[ \Gr_{G, \gencochar}(\FFresfield) =  G(\BdR^+)\gencochar(\gendRunif)^{-1}G(\BdR^+)/G(\BdR^+).\]
\end{prop}

\begin{remark}
In this paper, we won't use the language of diamonds in an essential way because we are only interested in the $\FFresfield$-valued points of $\Gr_G$ and $\Gr_{G, \gencochar}$.
\end{remark}

\begin{defn}
Let $\gencochar$ be a dominant cocharacter of $G$. 
\begin{enumerate}[label=(\arabic*)]
\item We refer to the locally spatial diamond $\Gr_{G, \gencochar}$ in Proposition \ref{schubert cells as diamonds} as the \emph{Schubert cell} of $\Gr_G$ associated to $\gencochar$. 
\smallskip

\item We define the \emph{parabolic subgroup of $G$ associated to $\gencochar$} by
\[ \genparabolic_\gencochar := \{ g \in G: \lim_{\gendRunif \to 0} \gencochar(t) g \gencochar(t)^{-1} \text{ exists} \}.\]

\item We define the \emph{flag variety} associated to the pair $(G, \gencochar)$ by
\[\flagvar(G, \gencochar) := G/\genparabolic_\gencochar.\]

\item We define the \emph{Bialynicki-Birula map} associated to $\gencochar$ as the map 
\[\BBmap{\gencochar}: \Gr_{G, \gencochar}(\FFresfield) \longrightarrow \flagvar(G, \gencochar)(\FFresfield)\] 
which associates to $g\gencochar(\gendRunif)^{-1}G(\BdR^+) \in \Gr_{G, \gencochar}(\FFresfield)$ the parabolic subgroup $\genred{g}\genparabolic_\gencochar\genred{g}^{-1}$, where $\genred{g}$ denotes the image of $g$ under the natural map $G(\BdR^+) \to G(\FFresfield)$. 
\end{enumerate}
\end{defn}


\begin{prop}[{\cite[Theorem 3.4.5]{CS_cohomunitaryshimura}}]\label{minuscule BB map is an isom}
If $\gencochar$ is a minuscule cocharacter of $G$, the Bialynicki-Birula map $\BBmap{\gencochar}$ is bijective. 
\end{prop}

\subsection{$G$-bundles on the Fargues-Fontaine curve}$ $

\begin{defn}\label{adicFFC} Fix a uniformizer $\uniformizer$ of $\genlocfield$ and a pseudouniformizer $\pseudounif$ of $\tilt{\FFresfield}$. Let $q$ be the number of elements in the residue field of $\genlocfield$. 

\begin{enumerate}[label=(\arabic*)]
\item We set
\[ \Ycal:= \Spa(\witt_{\integerring{\genlocfield}}(\integerring{\tilt{\FFresfield}})\setminus\{|\uniformizer [\pseudounif]|=0\},\]
where we write $\witt_{\integerring{\genlocfield}}(\integerring{\tilt{\FFresfield}}): = \witt(\integerring{\tilt{\FFresfield}}) \otimes_{\witt(\F_q)} \integerring{\genlocfield}$ for the ring of ramified Witt vectors over $\integerring{\tilt{\FFresfield}}$ with coefficients in $\integerring{\genlocfield}$ and the Teichmuller lift $[\pseudounif]$ of $\pseudounif$, and define the \emph{adic Fargues-Fontaine curve} associated to the pair $(\genlocfield, \tilt{\FFresfield})$ by
\[ \adicff:= \Ycal/\genfrob^\Z\]
where $\genfrob$ denotes the automorphism of $\Ycal$ induced by the $q$-Frobenius automorphism on $\witt_{\integerring{\genlocfield}}(\integerring{\tilt{\FFresfield}})$.
\smallskip

\item We define the \emph{schematic Fargues-Fontaine curve} associated to the pair $(\genlocfield, \tilt{\FFresfield})$ by
\[
\schff:= \Proj \left( \bigoplus_{n \geq 0 } H^0(\Ycal, \strsheaf[\Ycal])^{\genfrob = \uniformizer^n} \right).
\]
\end{enumerate}
\end{defn}

\begin{remark}
The definition of the adic Fargues-Fontaine curve relies on the fact that the action of $\genfrob$ on $\Ycal$ is properly discontinuous. 
\end{remark}

\begin{theorem}[{\cite[Theorem 4.10]{Kedlaya_noeth}}, {\cite[Th\'eor\`eme 6.5.2
]{FF_curve}}, {\cite[Theorem 8.7.7]{KL_relpadic1}}]\label{geom properties of FF curve} 
We have the following statements:
\begin{enumerate}[label=(\arabic*)]
\item $\adicff$ is a Noetherian adic space over $\genlocfield$. 

\item $\schff$ is a Dedekind scheme over $\genlocfield$. 
\smallskip

\item\label{GAGA for FF curve} There exists an equivalence of the categories of vector bundles on $\adicff$ and $\schff$, induced by pullback along a natural map of locally ringed spaces 
$\adicff \longrightarrow \schff$.
\end{enumerate}
\end{theorem}

\begin{remark}
The scheme $\schff$ is not a curve in the usual sense as it is not of finite type over 
$E$. 
\end{remark}

In light of the statement \ref{GAGA for FF curve} in Theorem \ref{geom properties of FF curve}, we will henceforth identify $G$-bundles on $\adicff$ with $G$-bundles on $\schff$.

\begin{defn}\label{G-bundle assoc to an elt of B(G)}
Given an element $b \in B(G)$, we define the associated $G$-bundle $\genvb_b$ on $\adicff$ (or on $\schff)$ by descending along the map $\Ycal \longrightarrow \Ycal/\genfrob^\Z = \adicff$ the trivial $G$-bundle on $\Ycal$ equipped with the $\genfrob$-linear automorphism given by $b$. 
\end{defn}


\begin{theorem}[{\cite[Th\'eor\`eme 5.1]{Fargues_Gbundle}}]\label{classification of G-bundles on FF curve}
The map $B(G) \longrightarrow H_{\et}^1(\schff, G)$ sending $b$ to the isomorphism class of $\genvb_b$ is a bijection. 
\end{theorem}

\begin{prop}\label{isocrystals and vector bundles on FF curve}
The set of isomorphism classes of isocrystals over $\compmaxunram{\genlocfield}$ and the set of isomorphism classes of vector bundles on $\schff$ admit a natural bijection which is compatible with direct sums, duals, and ranks.  
\end{prop}

\begin{proof}
Consider an arbitrary integer $n>0$. Given $b \in B(\GL_n)$, we write $\genisocrystal_b$ for the isocrystal over $\compmaxunram{\genlocfield}$ with underlying vector space $\compmaxunram{\genlocfield}^{\oplus n}$ and the Frobenius-semilinear automorphism given by $b$. As observed by Kottwitz \cite{Kottwitz_isocrystal}, there exists a natural bijection between $B(\GL_n)$ and the set of isomorphism classes of isocrystals over $\compmaxunram{\genlocfield}$ of rank $n$ where $b \in B(\GL_n)$ maps to the isomorphism class of $\genisocrystal_b$. Moreover, Theorem \ref{classification of G-bundles on FF curve} yields a bijection between $B(\GL_n)$ and the set of isomorphism classes of vector bundles over $\schff$ of rank $n$ where $b \in B(\GL_n)$ maps to the isomorphism class of $\genvb_b$. We thus obtain a bijection between the set of isomorphism classes of isocrystals over $\compmaxunram{\genlocfield}$ and the set of isomorphism classes of vector bundles on $\schff$. It is straight forward to check that this bijection is compatible with diret sums, duals, and ranks. 
\end{proof}

\begin{defn}\label{HN polygon of vector bundle on FF curve}
Let $\genvb$ be a vector bundle on $\schff$. We denote by $\genisocrystal(\genvb)$ the isomorphism class of isocrystals over $\compmaxunram{\genlocfield}$ that corresponds to $\genvb$ under the bijection in Proposition \ref{isocrystals and vector bundles on FF curve}. 
\begin{enumerate}[label=(\arabic*)]
\item We write $\rk(\genvb)$ for the rank of $\genvb$, and define the \emph{degree} of $\genvb$, denoted by $\deg(\genvb)$, to be the degree of $\genisocrystal(\genvb)$. 
\smallskip
\item We define the \emph{Harder-Narasimhan (HN) polygon} of $\genvb$ by $\HN(\genvb) := -\Newt(\genisocrystal(\genvb)^\vee)$, where $\Newt(\genisocrystal(\genvb)^\vee)$ refers to the Newton polygon of the dual of $\genisocrystal(\genvb)$.
\smallskip

\item We say that $\genvb$ is \emph{semistable} of slope $\genslope$ if $\HN(\genvb)$ is a line segment of slope $\genslope$. 
\end{enumerate}
\end{defn}

\begin{remark}
The definition of $\HN(\genvb)$ is in line with the convention 
that Newton polygons are convex while Harder-Narasimhan polygons are concave. 
It is also worthwhile to mention that the correct (or usual) definition of semistability should be given in terms of the Harder-Narasimhan formalism for vector bundles on $\schff$; in fact, the equivalence of our definition and the correct definition is due to a highly nontrivial result of Fargues-Fontaine \cite{FF_curve}. 
\end{remark}

\begin{prop}\label{HN decomp of vector bundles on FF curve}
Let $\genvb$ be a vector bundle on $\schff$.
\begin{enumerate}[label=(\arabic*)]
\item $\genvb$ admits a direct sum decomposition $\genvb \simeq \oplus \genvb_i$ where the $\genvb_i$'s are semistable vector bundles on $\schff$ of distinct slopes. 
\smallskip

\item\label{uniqueness of direct summands in HN decomp} If the $\genvb_i$'s are arranged in order of descending slope, $\HN(\genvb)$ is given by the concatenation of the polygons $\HN(\genvb_i)$. 
\end{enumerate}
\end{prop}

\begin{proof}
The assertion is evident by Proposition \ref{isocrystals and vector bundles on FF curve} and the semisimplicity of isocrystals. 
\end{proof}

\begin{remark}
The statement \ref{uniqueness of direct summands in HN decomp} implies that the direct summands $\genvb_i$ are uniquely determined up to permutations. 
\end{remark}

\begin{defn}
Let $\genvb$ be a vector bundle on $\schff$. We refer to the direct sum decomposition $\genvb \simeq \oplus \genvb_i$ in Proposition \ref{HN decomp of vector bundles on FF curve} as the \emph{Harder-Narasimhan (HN) decomposition} of $\genvb$. 
\end{defn}

\subsection{The Newton stratification of Schubert cells and flag varieties}$ $

For the rest of this paper, we fix a closed point $\FFclosedpt$ on $\schff$ given by the following proposition:
\begin{prop}[{\cite[Th\'eor\`emes 6.5.2 and 7.3.3]{FF_curve}, \cite[Remark 1.7]{CT_weakaddandnewton}}]\label{distinguished closed pt on FF curve}
There exists a closed point $\FFclosedpt$ on $\schff$ with the following properties:
\begin{enumerate}[label=(\roman*)]
\item $\schff - \FFclosedpt$ is the spectrum of a principal domain $\Be \subseteq \BdR$. 
\smallskip

\item The completed local ring at $\FFclosedpt$ is canonically isomorphic to $\BdR^+$. 
%
\end{enumerate}
\end{prop}

\begin{remark}
A closed point on $\schff$ corresponds to a characteristic $0$ untilt of $\tilt{\FFresfield}$ (i.e., a perfectoid field $\genpadicfield$ with an isomorphism $\tilt{\genpadicfield} \simeq \tilt{\FFresfield}$) up to $\genfrob$-equivalences. 
We may take $\FFclosedpt$ to be the closed point on $\schff$ corresponding to $\FFresfield$ with the identity map on $\tilt{\FFresfield}$. The field $\FFresfield$ alone does not determine $\FFclosedpt$ as $\tilt{\FFresfield}$ has automorphisms which are not $\genfrob$-equivalent to the identity map. 
\end{remark}


\begin{prop}\label{Beauville-Laszlo for FF curve}
The set $H_{\et}^1(\schff, G)$ is naturally in bijection with the set of isomorphism classes of triples $(\punctuation{\genvb}, \completion{\genvb}, \gentrivialization)$ where
\begin{itemize}
\item $\punctuation{\genvb}$ is a $G$-bundle on $\schff - \FFclosedpt$,
\smallskip

\item $\completion{\genvb}$ is a trivial $G$-bundle on $\Spec(\BdR^+)$, and
\smallskip

\item $\gentrivialization$ is a gluing map of $\punctuation{\genvb}$ and $\completion{\genvb}$ over $\Spec(\BdR)$. 
\end{itemize}
\end{prop}
\begin{proof}
Every $G$-bundle on $\schff$ becomes trivial after the pullback via the map $\Spec(\BdR^+) \to \schff$ induced by $\FFclosedpt$, as noted by Nguyen-Viehmann \cite[\S2.1]{NV_HNstrata} and Chen-Tong \cite[Remark 1.7]{CT_weakaddandnewton}. Hence the desired assertion follows from Proposition \ref{distinguished closed pt on FF curve} and the theorem of Beauville-Laszlo \cite{BL_beauvillelaszlothm}. 
\end{proof}

\begin{defn}\label{modification of G-bundle on FF curve}
Let $\genvb$ be a $G$-bundle on $\schff$. A \emph{modification} of $\genvb$ at $\FFclosedpt$ is a $G$-bundle $\modif{\genvb}$ on $\schff$ together with an isomorphism between $\genvb$ and $\modif{\genvb}$ on $\schff - \FFclosedpt$. 
\end{defn}

\begin{example}\label{modification given by a point on BdR+ Grassmannian}
Consider an element $b \in B(G)$ and a point $x \in \Gr_G(\FFresfield)$. We may write $x = gG(\BdR^+)$ for some $g \in G(\BdR)$ under the identification $\Gr_G(\FFresfield) \cong G(\BdR)/G(\BdR^+)$ noted in Proposition \ref{points of BdR-Grassmannian}. Now, in light of Proposition \ref{Beauville-Laszlo for FF curve} we take a triple $(\punctuation{\genvb}, \completion{\genvb}, \gentrivialization)$ corresponding to $\genvb_b$ and a $G$-bundle $\genvb_{b, x}$ on $\schff$ corresponding to $(\punctuation{\genvb}, \completion{\genvb}, g\gentrivialization)$. 
By construction, 
$\genvb_{b, x}$ is naturally a modification of $\genvb_b$ at $\FFclosedpt$. 
\end{example}

\begin{defn}\label{def of newton stratification}
Consider an element $b \in B(G)$ and a dominant cocharacter $\gencochar$ of $G$. 
\begin{enumerate}[label=(\arabic*)]
\item For each $x \in \Gr_G(\FFresfield)$, we refer to the $G$-bundle $\genvb_{b, x}$ constructed in Example \ref{modification given by a point on BdR+ Grassmannian} as the \emph{modification of $\genvb_b$ at $\FFclosedpt$ induced by $x$}. 
\smallskip

\item For each $\modif{b} \in B(G)$, we define the associated \emph{Newton stratum with respect to $b$} in $\Gr_{G, \gencochar}$ as the subdiamond $\Gr_{G, \gencochar, b}^{\modif{b}}$ of $\Gr_{G, \gencochar}$ with
\[\Gr_{G, \gencochar, b}^{\modif{b}}(\FFresfield) = \{x \in \Gr_{G, \gencochar}(\FFresfield): \genvb_{b, x} \simeq \genvb_{\modif{b}}\}.\]

\item For each $\modif{b} \in B(G)$, we define the associated \emph{Newton stratum with respect to $b$} in $\flagvar(G, \gencochar)$ as the subvariety $\flagvar(G, \gencochar, b)^{\modif{b}}$ of $\flagvar(G, \gencochar)$ such that $\flagvar(G, \gencochar, b)^{\modif{b}}(\FFresfield)$ is the image of $\Gr_{G, \gencochar, b}^{\modif{b}}(\FFresfield)$ under the map $\BBmap{\gencochar}$. 
\end{enumerate}
\end{defn}

\begin{remark}
The subdiamond $\Gr_{G, \gencochar, b}^{\modif{b}}$ of $\Gr_{G, \gencochar}$ is uniquely determined by its set of $\FFresfield$-points since $\Gr_{G, \gencochar}$ is a locally spatial diamond.
\end{remark}

\subsection{Subsheaves and extensions of vector bundles on the Fargues-Fontaine curve}$ $

\begin{defn}
Given two integers $n$ and $d$ with $n >0$, a \emph{rationally tuplar polygon} of rank $n$ and degree $d$ is the graph $\genpolygon$ of a continuous function $f$ with the following properties:
\begin{enumerate}[label=(\roman*)]
\item $f$ is defined on $[0, n]$ with $f(0) = 0$ and $f(n) = d$. 
\smallskip

\item $f$ is linear on $[i-1, i]$ for each $i = 1, \cdots, n$ with a rational slope denoted by $\HNslope{\genpolygon}{i}$. 
\end{enumerate} 
\end{defn}

\begin{example}\label{examples of rationally tuplar polygon}
We are particularly interested in the following rationally tuplar polygons:
\begin{enumerate}[label=(\arabic*)]
\item For every vector bundle $\genvb$ on $\schff$ of rank $n$ and degree $d$, its HN polygon $\HN(\genvb)$ is a rationally tuplar polygon of rank $n$ and degree $d$.
\smallskip

\item For $G = \GL_n$ with Borel subgroup $\genborel$ of upper triangular matrices and maximal torus $\genmaxtorus$ of diagonal matrices, we regard all dominant cocharacters as rationally tuplar polygons of rank $n$ under the natural identification
\[\domcocharset \cong \{ (a_i) \in \Z^n: a_1 \geq a_2 \geq \cdots \geq a_n\}.\]

\item We write $\lineseg{d/n}{n}$ for the line segment connecting $(0, 0)$ and $(n, d)$, which is a rationally tuplar polygon of rank $n$ and degree $d$.
\end{enumerate}
\end{example}

\begin{defn}\label{partial orders for rationally tuplar polygons}
Let $\polyset{n}$ denote the set of rationally tuplar polygons of rank $n$. 
\begin{enumerate}[label=(\arabic*)]
\item We define the \emph{Bruhat order} $\geq$ on $\polyset{n}$ by writing $\genpolygon \geq \gensecpolygon$ if we have 
\[ \sum_{i = 1}^j \HNslope{\genpolygon}{i} \geq \sum_{i=1}^j \HNslope{\gensecpolygon}{i} \quad\quad \text{ for each } j = 1, \cdots, n\]
with equality for $j = n$.
\smallskip

\item We define the \emph{slopewise dominance order} $\slodom$ on $\polyset{n}$ by writing $\genpolygon \slodom \gensecpolygon$ if we have $\HNslope{\genpolygon}{i} \geq \HNslope{\gensecpolygon}{i}$ for each $i = 1, \cdots, n$. 
\end{enumerate}
\end{defn}

\begin{remark}
Intuitively, we have $\genpolygon \geq \gensecpolygon$ if and only if $\genpolygon$ lies on or above $\gensecpolygon$ with the same endpoints, as illustrated by Figure \ref{dominance order figure illustration}.
\end{remark}
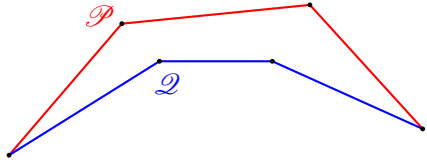
\begin{figure}[H]
\begin{tikzpicture}[scale=0.5]	
		\coordinate (left) at (0, 0);
		\coordinate (p1) at (3, 3.5);
		\coordinate (p2) at (8, 4);
		\coordinate (right) at (11, 0.7);

		\coordinate (q2) at (4, 2.5);
		\coordinate (q3) at (7, 2.5);
				
		\draw[step=1cm,thick, color=red] (left) -- (p1) --  (p2) -- (right);
		\draw[step=1cm,thick, color=blue] (left) -- (q2) -- (q3) -- (right);

		\draw [fill] (q2) circle [radius=0.05];		
		\draw [fill] (q3) circle [radius=0.05];			
		\draw [fill] (left) circle [radius=0.05];
		\draw [fill] (right) circle [radius=0.05];		
		
		\draw [fill] (p1) circle [radius=0.05];		
		\draw [fill] (p2) circle [radius=0.05];		
		
		\path (q2) ++(0.2, -0.6) node {\color{blue}$\gensecpolygon$};
		\path (p1) ++(-0.6, 0.2) node {\color{red}$\genpolygon$};
\end{tikzpicture}
\caption{Illustration of the Bruhat order}\label{dominance order figure illustration}
\end{figure}

\begin{prop}[{\cite[Theorem 1.2.1]{Hong_subvb}}]\label{subsheaves and slopewise dominance}
Let $\gensubvb$ and $\genvb$ be vector bundles on $\schff$ of rank $n$. 
Then $\gensubvb$ is a subsheaf of $\genvb$ if and only if we have $\HN(\genvb) \slodom \HN(\gensubvb)$. 
\end{prop}

\begin{defn}\label{suitable permutation of HN polygon for extension}
Given vector bundles $\gensubvb$, $\genextvb$ and $\genquotvb$ on $\schff$, we define 
a \emph{$(\gensubvb, \genextvb, \genquotvb)$-permutation} of $\HN(\gensubvb \oplus \genquotvb)$ to be a rationally tuplar polygon $\genpolygon \geq \HN(\genextvb)$ with the following properties:
\begin{enumerate}[label=(\roman*)]
\item The tuple $(\HNslope{\genpolygon}{i})$ is a permutation of the tuple $(\HNslope{\HN(\gensubvb \oplus \genquotvb)}{i})$. 
\smallskip


\item For each $i = 1, \cdots, \rk(\genextvb)$, we have
\smallskip
\begin{itemize}
\item $\HNslope{\genpolygon}{i} < \HNslope{\HN(\genextvb)}{i}$ only if $\HNslope{\genpolygon}{i}$ occurs as a slope in $\HN(\gensubvb)$, and
\smallskip

\item $\HNslope{\genpolygon}{i} > \HNslope{\HN(\genextvb)}{i}$ only if $\HNslope{\genpolygon}{i}$ occurs as a slope in $\HN(\genquotvb)$.
\end{itemize}
\end{enumerate}
\end{defn}

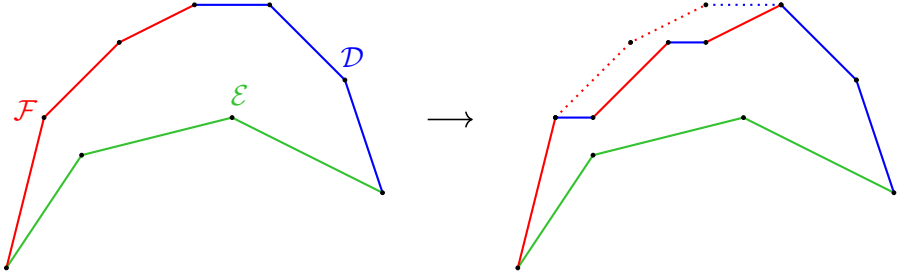
\begin{figure}[H]
\begin{tikzpicture}[scale=0.5]	
		\coordinate (left) at (0, 0);
		\coordinate (p1) at (2, 3);
		\coordinate (p2) at (6, 4);
		\coordinate (right) at (10, 2);

		\coordinate (q1) at (1, 4);
		\coordinate (q2) at (3, 6);
		\coordinate (q3) at (5, 7);
		\coordinate (q4) at (7, 7);
		\coordinate (q5) at (9, 5);
				
		\draw[step=1cm,thick, color=mynicegreen] (left) -- (p1) --  (p2) -- (right);
		\draw[step=1cm,thick, color=red] (left) -- (q1) -- (q2) -- (q3);
		\draw[step=1cm,thick, color=blue] (q3) -- (q4) -- (q5) -- (right);

		\draw [fill] (q1) circle [radius=0.05];		
		\draw [fill] (q2) circle [radius=0.05];		
		\draw [fill] (q3) circle [radius=0.05];		
		\draw [fill] (q4) circle [radius=0.05];		
		\draw [fill] (q5) circle [radius=0.05];		
		\draw [fill] (left) circle [radius=0.05];
		\draw [fill] (right) circle [radius=0.05];		
		
		\draw [fill] (p1) circle [radius=0.05];		
		\draw [fill] (p2) circle [radius=0.05];		
		
		\path (q5) ++(0.2, 0.6) node {\color{blue}$\gensubvb$};
		\path (p2) ++(0.2, 0.6) node {\color{mynicegreen}$\genextvb$};
		\path (q1) ++(-0.5, 0.2) node {\color{red}$\genquotvb$};
\end{tikzpicture}
\hspace{0.3cm}
\begin{tikzpicture}[scale=0.4]
        \pgfmathsetmacro{\textycoordinate}{5}
		\draw[->, line width=0.6pt] (0, \textycoordinate) -- (1.5,\textycoordinate);
		\draw (0,0) circle [radius=0.00];	
\end{tikzpicture}
\hspace{0.3cm}
\begin{tikzpicture}[scale=0.5]
		\coordinate (left) at (0, 0);
		\coordinate (p1) at (2, 3);
		\coordinate (p2) at (6, 4);
		\coordinate (right) at (10, 2);

		\coordinate (q1) at (1, 4);
		\coordinate (q2) at (3, 6);
		\coordinate (q3) at (5, 7);
		\coordinate (q4) at (7, 7);
		\coordinate (q5) at (9, 5);
		
		\coordinate (r1) at (q1);
		\coordinate (r2) at (2, 4);
		\coordinate (r3) at (4, 6);
		\coordinate (r4) at (5, 6);
		\coordinate (r5) at (q4);
		\coordinate (r6) at (q5);
				
		\draw[step=1cm,thick, color=mynicegreen] (left) -- (p1) --  (p2) -- (right);
		\draw[step=1cm,thick, color=red] (left) -- (q1);
		\draw[step=1cm,thick,dotted, color=red] (q1) -- (q2) -- (q3);
		\draw[step=1cm,thick,dotted, color=blue] (q3) -- (q4);
		\draw[step=1cm,thick, color=blue] (q4) -- (q5) -- (right);
		
		\draw[step=1cm,thick, color=blue] (r1) -- (r2);
		\draw[step=1cm,thick, color=red] (r2) -- (r3);
		\draw[step=1cm,thick, color=blue] (r3) -- (r4);
		\draw[step=1cm,thick, color=red] (r4) -- (r5);

		\draw [fill] (q1) circle [radius=0.05];		
		\draw [fill] (q2) circle [radius=0.05];		
		\draw [fill] (q3) circle [radius=0.05];		
		\draw [fill] (q4) circle [radius=0.05];		
		\draw [fill] (q5) circle [radius=0.05];		
		\draw [fill] (left) circle [radius=0.05];
		\draw [fill] (right) circle [radius=0.05];		
		
		\draw [fill] (p1) circle [radius=0.05];		
		\draw [fill] (p2) circle [radius=0.05];	

		\draw [fill] (r1) circle [radius=0.05];		
		\draw [fill] (r2) circle [radius=0.05];		
		\draw [fill] (r3) circle [radius=0.05];		
		\draw [fill] (r4) circle [radius=0.05];		
		\draw [fill] (r5) circle [radius=0.05];		
		\draw [fill] (r6) circle [radius=0.05];		
		
\end{tikzpicture}
\caption{Illustration of the conditions in Definition \ref{suitable permutation of HN polygon for extension}}
\end{figure}

\begin{prop}[{\cite[Proposition 5.6.23]{FF_curve}}]\label{splitting extension for vector bundles with dominating slopes}
Given vector bundles $\gensubvb$ and $\genquotvb$ on $\schff$ such that the minimum slope in $\HN(\gensubvb)$ is greater than or equal to the maximum slope in $\HN(\genquotvb)$, every extension of $\genquotvb$ by $\gensubvb$ splits. 
\end{prop}

\begin{prop}[{\cite[Theorem 3.12]{Hong_extvbfinal}}, {\cite[Proposition 5.3]{CT_weakaddandnewton}}]\label{classification of extensions, necessity condition on E-permutation}
Let $\gensubvb$, $\genextvb$ and $\genquotvb$ be vector bundles on $\schff$ such that there exists a short exact sequence 
\begin{equation*}\label{short exact sequence, semistable case}
0 \longrightarrow \gensubvb \longrightarrow \genextvb \longrightarrow \genquotvb \longrightarrow 0.
\end{equation*}
There exists a $(\gensubvb, \genextvb, \genquotvb)$-permutation of $\HN(\gensubvb \oplus \genquotvb)$.
\end{prop}

\begin{prop}[{\cite[Theorem 4.4]{Hong_extvbfinal}}, {\cite[Proposition 5.9]{CT_weakaddandnewton}}]\label{classification of extensions}
Let $\gensubvb$, $\genextvb$ and $\genquotvb$ be vector bundles on $\schff$. We write the HN decomposition of $\genquotvb$ as
\[ \genquotvb \simeq \bigoplus_{i=1}^\numslope \genquotvb_i\]
where the $\genquotvb_i$'s are arranged in order of descending slope. There exists a short exact sequence 
\begin{equation*}\label{short exact sequence, semistable case}
0 \longrightarrow \gensubvb \longrightarrow \genextvb \longrightarrow \genquotvb \longrightarrow 0
\end{equation*}
if and only if there exists a sequence of vector bundles $\gensubvb = \genvb_0, \genvb_1, \cdots, \genvb_\numslope = \genvb$ on $\schff$ such that the polygon $\HN(\genvb_{i-1} \oplus \genquotvb_i)$ has an $(\genextvb_{i-1}, \genextvb_i, \genquotvb_i)$-permutation for each $i = 1, \cdots, \numslope$. 
\end{prop}

\section{Nonempty Newton strata in minuscule Schubert cells for $\GL_n$}

In this section, we classify all nonempty Newton strata in an arbitrary minuscule Schubert cell for $\GL_n$ by studying modifications of vector bundles on the Fargues-Fontaine curve. We first establish in \S\ref{subsection for inductive classification of nonempty newton strata} an inductive classification for nonempty Newton strata associated to an arbitrary element of $B(\GL_n)$. We then prove in \S\ref{combinatorics on polygons} some combinatorial lemmas about rationally tuplar polygons and use them in \S\ref{subsection for explicit classification of Newton strata} to give an explicit classification of all nonempty Newton strata associated to a large class of element of $B(\GL_n)$. Throughout this section, we take dominant cocharacters of $\GL_n$ with respect to the standard Borel subgroup of upper triangular matrices and the standard maximal torus of diagonal matrices.

\subsection{An inductive classification of nonempty Newton strata}\label{subsection for inductive classification of nonempty newton strata}$ $

\begin{defn}
Given a rationally tuplar polygon $\genpolygon$ of rank $n$, we define its \emph{dual} to be the rationally tuplar polygon $\dualpolygon{\genpolygon}$ with $\HNslope{\dualpolygon{\genpolygon}}{i} = -\HNslope{\genpolygon}{n+1-i}$ for each $i = 1, \cdots, n$. 
\end{defn}

\begin{example}\label{examples of dual polygons}
We illustrate the notion of duality for the polygons in Example \ref{examples of rationally tuplar polygon}. 
\begin{enumerate}[label=(\arabic*)]
\item For a vector bundle $\genvb$ on $\schff$ of rank $n$, we have $\dualpolygon{\HN(\genvb)} = \HN(\genvb^\vee)$ where $\genvb^\vee$ denotes the dual bundle of $\genvb$. 
\smallskip

\item For a dominant cocharacter $\gencochar$ of $\GL_n$, the polygon $\dualpolygon{\gencochar}$ represents the unique dominant cocharacter in the conjugacy class of $\gencochar^{-1}$. 
\smallskip

\item For arbitrary integers $d$ and $n$, we have $\dualpolygon{\lineseg{d/n}{n}} = \lineseg{-d/n}{n}$. 
\end{enumerate}
\end{example}

\begin{prop}[{\cite[Proposition 5.2]{CFS_admlocus}, \cite[Corollary 5.4]{Viehmann_weakadmlocNewton}}]\label{classification of nonempty newton strata with respect to basic element}
Let $b$ and $\modif{b}$ be elements of $B(\GL_n)$ such that $\genvb_b$ is semistable. Given a dominant cocharacter $\gencochar$ of $\GL_n$, the Newton stratum $\Gr_{\GL_n, \gencochar, b}^{\modif{b}}$ is nonempty if and only if we have 
\begin{equation}\label{newton map inequality for semistable case}
\newtonmap{b} + \dualpolygon{\gencochar} \geq \newtonmap{\modif{b}}
\end{equation}
where $\newtonmap{b}$ and $\newtonmap{\modif{b}}$ respectively denote $\HN(\genvb_b)$ and $\HN(\genvb_{\modif{b}})$.  
\end{prop}

\begin{remark}
For a reductive group $G$ and a basic element $b \in B(G)$, the results of Chen-Fargues-Shen \cite[Proposition 5.2]{CFS_admlocus} and Viehmann \cite[Corollary 5.4]{Viehmann_weakadmlocNewton} classify all nonempty newton strata with respect to $b$ in an arbitrary Schubert cell in terms of the Kottwitz map and the Newton map defined by Kottwitz \cite{Kottwitz_isocrystal}. In our context, their results are translated to Proposition \ref{classification of nonempty newton strata with respect to basic element} by the following facts:
\begin{enumerate}[label=(\alph*)]
\item An element $b \in B(\GL_n)$ is basic if and only if $\genvb_b$ is semistable. 
\smallskip

\item The condition involving the Kottwitz map holds for all elements in $B(\GL_n)$. 
\smallskip

\item The condition involving the Newton map is equivalent to the inequality \eqref{newton map inequality for semistable case} as $\newtonmap{b}$ and $\newtonmap{\modif{b}}$ are identified with the (concave) Newton polygons of $b$ and $\modif{b}$. 
\end{enumerate}
\end{remark}

\begin{lemma}\label{Newton stratum for central twist}
Let $b$ be an element of $B(\GL_n)$. For $x = \lineseg{1}{n}(\gendRunif) \GL_n(\BdR^+) \in \Gr_{\GL_n, \lineseg{1}{n}}(\FFresfield)$, we have $\HN(\genvb_{b, x}) = \HN(\genvb_b) - \lineseg{1}{n}$. 
\end{lemma}

\begin{proof}
Let us write the HN decomposition of $\genvb_b$ as
\[ \genvb_b \simeq \bigoplus_{i=1}^\numslope \genvb_{b_i} \quad \text{ with } b_i \in B(\GL_{n_i}).\]
For each $i = 1, \cdots, \numslope$, we take $x_i:= \lineseg{1}{n_i}(\gendRunif) \GL_{n_i}(\BdR^+) \in \Gr_{\GL_{n_i}, \lineseg{1}{n_i}}(\FFresfield)$. 
Then we have $\HN(\genvb_{b_i, x_i}) \leq \HN(\genvb_{b_i}) - \lineseg{1}{n_i}$ by Proposition \ref{classification of nonempty newton strata with respect to basic element}
and thus find $\HN(\genvb_{b_i, x_i}) = \HN(\genvb_{b_i}) - \lineseg{1}{n_i}$ as $\HN(\genvb_{b_i}) - \lineseg{1}{n_i}$ is a line segment. Now the desired assertion follows by the fact that $\genvb_{b, x}$ is a direct sum of the vector bundles $\genvb_{b_i, x_i}$. 
\end{proof}

\begin{prop}\label{nonemptiness of newton strata via points}
Let $\gencochar$ be a dominant cocharacter of $\GL_n$. For elements $b, \modif{b} \in B(\GL_n)$, the Newton stratum $\Gr_{\GL_n, \gencochar, b}^{\modif{b}}$ is not empty if and only if it contains a $\FFresfield$-point. 
\end{prop}

\begin{proof}
The assertion is evident by definition. 
\end{proof}

\begin{prop}\label{reduction to minimal minuscule}
Let $\gencochar$ be a dominant cocharacter of $\GL_n$ with nonnegative slopes. For two elements $b, \modif{b} \in B(\GL_n)$, we have the following equivalent conditions:
\begin{enumerate}[label=(\roman*)]
\item\label{nonemptiness of original stratum} $\Gr_{\GL_n, \gencochar, b}^{\modif{b}}$ is nonempty. 
\smallskip

\item\label{nonemptiness of dual stratum} $\Gr_{\GL_n, \dualpolygon{\gencochar}, \modif{b}}^{b}$ is nonempty. 
\smallskip

\item\label{nonemptiness of shifted stratum} $\Gr_{\GL_n, \gencochar+\lineseg{1}{n}, b}^{\genlift{\modif{b}}}$ is nonempty for $\genlift{\modif{b}} \in B(\GL_n)$ with $\HN(\genvb_{\genlift{\modif{b}}}) = \HN(\genvb_{\modif{b}}) - \lineseg{1}{n}$. 
\end{enumerate}
\end{prop}

\begin{proof}
For $x = g\gencochar(\gendRunif)G(\BdR^+) \in \Gr_{\GL_n, \gencochar, b}^{\modif{b}}(\FFresfield)$, we take $\dualpolygon{x} := g^{-1}\dualpolygon{\gencochar}(\gendRunif)G(\BdR^+) \in \Gr_{\GL_n, \dualpolygon{\gencochar}}(\FFresfield)$ and find $\genvb_{\modif{b}, \dualpolygon{x}} \simeq \genvb_b$, thereby deducing that $\dualpolygon{x}$ lies in $\Gr_{\GL_n, \dualpolygon{\gencochar}, \modif{b}}^{b}(\FFresfield)$. Similarly, every point in $\Gr_{\GL_n, \dualpolygon{\gencochar}, \modif{b}}^{b}(\FFresfield)$ gives rise to a point in $\Gr_{\GL_n, \gencochar, b}^{\modif{b}}(\FFresfield)$. Hence by Proposition \ref{nonemptiness of newton strata via points} we establish the equivalence of the conditions \ref{nonemptiness of original stratum} and \ref{nonemptiness of dual stratum}.

Now it remains to verify the equivalence of the conditions \ref{nonemptiness of original stratum} and \ref{nonemptiness of shifted stratum}. For every $x = g\gencochar(\gendRunif)G(\BdR^+) \in \Gr_{\GL_n, \gencochar, b}^{\modif{b}}(\FFresfield)$, we take $\genlift{x} := g\gencochar(\gendRunif)\lineseg{1}{n}(\gendRunif) G(\BdR^+) \in \Gr_{\GL_n, \gencochar+\lineseg{1}{n}}(\FFresfield)$ and find $\genvb_{b, \genlift{x}} \simeq \genvb_{\genlift{\modif{b}}}$ by Lemma \ref{Newton stratum for central twist}, thereby deducing that $\genlift{x}$ lies in $\Gr_{\GL_n, \gencochar+\lineseg{1}{n}, b}^{\genlift{\modif{b}}}(\FFresfield)$. Conversely, for every $\genlift{x} := g\gencochar(\gendRunif)\lineseg{1}{n}(\gendRunif) G(\BdR^+) \in \Gr_{\GL_n, \gencochar+\lineseg{1}{n}}(\FFresfield)$ we take $x := g\gencochar(\gendRunif)G(\BdR^+) \in \Gr_{\GL_n, \gencochar, b}^{\modif{b}}(\FFresfield)$ and find $\genvb_{b, x} \simeq \genvb_{\modif{b}}$ by Lemma \ref{Newton stratum for central twist}, thereby deducing that $x$ lies in $\Gr_{\GL_n, \gencochar, b}^{\modif{b}}(\FFresfield)$. Hence we complete the proof by Proposition \ref{nonemptiness of newton strata via points}. 
\end{proof}

\begin{remark}
In light of Proposition \ref{reduction to minimal minuscule}, for our desired classification it suffices to consider minuscule cocharacters with slopes $0$ and $1$. 
\end{remark}


\begin{defn}
Let $\genvb$ be a vector bundle on $\schff$ of rank $n$.
\begin{enumerate}[label=(\arabic*)]
\item Given a dominant cocharacter $\gencochar$ of $\GL_n$, we define an \emph{effective modification} of $\genvb$ at $\FFclosedpt$ of \emph{type $\gencochar$} to be an injective $\strsheaf[\schff]$-module map $\modif{\genvb} \inj \genvb$ whose cokernel is the skyscraper sheaf at $\FFclosedpt$ with value $\displaystyle \bigoplus_{i=1}^n \BdR^+/ \gendRunif^{\HNslope{\gencochar}{i}} \BdR^+$. 
\smallskip

\item We say that an effective modification $\modif{\genvb} \inj \genvb$ at $\FFclosedpt$ is \emph{minuscule of degree $d$} if its type is minuscule of degree $d$ with slopes $0$ and $1$. 
\end{enumerate} 
\end{defn}

\begin{prop}\label{minuscule newton strata and minuscule modification}
Take a dominant cocharacter $\gencochar$ of $\GL_n$ and two elements $b, \modif{b} \in B(\GL_n)$. 
\begin{enumerate}[label=(\arabic*)]
\item If $\gencochar$ has nonnegative slopes, the Newton stratum $\Gr_{\GL_n, \gencochar, b}^{\modif{b}}$ is nonempty if and only if there exists an effective modification $\genvb_{\modif{b}} \inj \genvb_b$ at $\FFclosedpt$ of type $\gencochar$. 
\smallskip

\item If $\gencochar$ is minuscule with slopes $0$ and $1$, the Newton stratum $\Gr_{\GL_n, \gencochar, b}^{\modif{b}}$ is nonempty if and only if there exists a minuscule effective modification $\genvb_{\modif{b}} \inj \genvb_b$ at $\FFclosedpt$. 
\end{enumerate}
\end{prop}
\begin{proof}
As the second statement is a special case of the first statement, it suffices to prove the first statement. 
If $\Gr_{\GL_n, \gencochar, b}^{\modif{b}}$ is not empty, Proposition \ref{nonemptiness of newton strata via points} yields a point $x \in \Gr_{\GL_n, \gencochar, b}^{\modif{b}}(\FFresfield)$, which 
gives rise to an effective modification $\genvb_{b, x} \inj \genvb_b$ at $\FFclosedpt$ of type $\gencochar$. Let us now assume for the converse that there exists an effective modification $\genvb_{\modif{b}} \inj \genvb_b$ at $\FFclosedpt$ of type $\gencochar$. Take triples $(\punctuation{\genvb}_b, \completion{\genvb_b}, \gentrivialization_b)$ and $(\punctuation{\genvb}_{\modif{b}}, \completion{\genvb_{\modif{b}}}, \gentrivialization_{\modif{b}})$ which respectively correspond to $\genvb_b$ and $\genvb_{\modif{b}}$ under the bijection in Proposition \ref{Beauville-Laszlo for FF curve}. We may set $\punctuation{\genvb}_{b} = \punctuation{\genvb}_{\modif{b}}$ since the map $\genvb_{\modif{b}} \inj \genvb_b$ is an isomorphism on $\schff - \FFclosedpt$. Then we conjugate $\gentrivialization_b$ by a suitable element in $\GL_n(\BdR^+)$ to write $\gentrivialization_{\modif{b}} = g\gencochar(t)\gentrivialization_b$ for some $g \in \GL_n(\BdR^+)$, and in turn find $g\gencochar(t)\GL_n(\BdR^+) \in \Gr_{\GL_n, \gencochar, b}^{\modif{b}}(\FFresfield)$ to complete the proof. 
\end{proof}


\begin{prop}\label{reduction lemma for minuscule modifications}
Let $\genvb$ and $\modif{\genvb}$ be vector bundles on $\schff$ of rank $n$. 
Take a direct sum decomposition
\begin{equation}\label{reduction lemma direct decomposition with a semistable factor} 
\genvb \simeq \gensubvb \oplus \genquotvb
\end{equation}
such that $\HN(\gensubvb)$ coincides with the line segment of maximal slope in $\HN(\genvb)$. 
There exists a minuscule effective modification $\modif{\genvb} \inj \genvb$ at $\FFclosedpt$ if and only if there exist minuscule effective modifications $\modif{\gensubvb} \inj \gensubvb$ and $\modif{\genquotvb} \inj \genquotvb$ at $\FFclosedpt$ with a short exact sequence
\[0 \longrightarrow \modif{\gensubvb} \longrightarrow \modif{\genvb} \longrightarrow \modif{\genquotvb} \longrightarrow 0.\]
%
\end{prop}

\begin{proof}
The assertion is essentially a result of Chen-Tong \cite[Proposition 4.6]{CT_weakaddandnewton}. Our main observation is that, while the result in loc. cit. for $G = \GL_n$ only concerns the case where $\genvb$ is semistable, its proof remains valid without the semistability assumption on $\genvb$. For convenience of the readers, we explain how the result in loc. cit. is translated to the desired assertion. 

Let us take $b, \modif{b} \in B(\GL_n)$ with $\genvb \simeq \genvb_b$ and $\modif{\genvb} \simeq \genvb_{\modif{b}}$. We write $r$ for the rank of $\gensubvb$ and $\genparabolic$ for the standard parabolic subgroup of $\GL_n$ with Levi subgroup
\[\genlevi := \GL_r \times \GL_{n-r} \subseteq \GL_n.\]
The direct sum decomposition \eqref{reduction lemma direct decomposition with a semistable factor} corresponds to an element $b_\genlevi \in B(\genlevi)$ which maps to $b$ under the natural map $B(\genlevi) \longrightarrow B(G)$. 
Let $E(\genlevi, \modif{b})$ denote the set of all elements $\modif{b}_\genlevi \in B(\genlevi)$ which correspond to a direct sum $\modif{\gensubvb} \oplus \modif{\genquotvb}$ for some vector bundles $\modif{\gensubvb}$ of rank $r$ and $\modif{\genquotvb}$ of rank $n-r$ such that $\modif{\genvb} \simeq \genvb_{\modif{b}}$ arises as an extension of $\modif{\genquotvb}$ by $\modif{\gensubvb}$. 

We take $\gencochar$ to be the minuscule dominant cocharacter of $\GL_n$ of degree $d:= \deg(\genvb) - \deg(\modif{\genvb})$ with slopes $0$ and $1$. In addition, we choose an arbitrary element $w$ in the Weyl group of $\GL_n$ and denote by $\gencochar^w$ the dominant cocharacter of $\genlevi$ whose $\genlevi$-conjugacy class contains the $w$-conjugate of $\gencochar$. 
We have $\gencochar^w = (\gencochar_1, \gencochar_2)$ for some minuscule dominant cocharacters $\gencochar_1$ of $\GL_r$ and $\gencochar_2$ of $\GL_{n-r}$. We denote the degrees of $\gencochar_1$ and $\gencochar_2$ respectively by $d_1$ and $d_2$. 

Let $\flagvar(\GL_n, \gencochar)_\genparabolic^w$ be the subscheme of $\flagvar(\GL_n, \gencochar)$ given by the 
$\genparabolic$-orbit of $\genparabolic_{\gencochar^w}$. 
The projection to $\genlevi$ induces a map
\[ \mathrm{pr}_{\genparabolic, w}: \flagvar(\GL_n, \gencochar)_\genparabolic^w \longrightarrow \flagvar(\genlevi, \gencochar^w).\]
The aforementioned result of Chen-Tong \cite[Proposition 4.6]{CT_weakaddandnewton} yields an identity
\begin{equation}\label{Chen-Tong identity}
\mathrm{pr}_{\genparabolic, w}\left(\flagvar(\GL_n, \gencochar)_\genparabolic^w \cap \flagvar(\GL_n, \gencochar, b)^{\modif{b}}\right) = \bigsqcup_{\modif{b}_\genlevi \in E(\genlevi, \modif{b})} \flagvar(\genlevi, \gencochar^w, b_\genlevi)^{\modif{b}_\genlevi}.
\end{equation}
%
%
As both $\gencochar$ and $\gencochar^w$ are minuscule, Proposition \ref{minuscule BB map is an isom} implies that the Newton strata on $\Gr_{\GL_n, \gencochar}$ and $\Gr_{\genlevi, \gencochar^w}$ are respectively identified with the Newton strata on $\flagvar(\GL_n, \gencochar)$ and $\flagvar(\genlevi, \gencochar^w)$. 
Hence the identity \eqref{Chen-Tong identity} shows that for minuscule effective modifications $\alpha: \modif{\gensubvb} \inj \gensubvb$ and $\beta: \modif{\genquotvb} \inj \genquotvb$ at $\FFclosedpt$ of degrees $d_1$ and $d_2$ we have the following equivalent conditions:
\begin{enumerate}[label=(\roman*)]
\item There exists a
commutative diagram of short exact sequences
\begin{equation}\label{extension diagram of minuscule modifications}
\begin{tikzcd}
0 \arrow[r]& \gensubvb \arrow[r] & \genvb \arrow[r]& \genquotvb \arrow[r]& 0\\
0 \arrow[r]& \modif{\gensubvb} \arrow[r] \arrow[u, "\alpha", hookrightarrow] & \modif{\genvb} \arrow[r] \arrow[u, hookrightarrow] & \modif{\genquotvb} \arrow[r] \arrow[u, "\beta", hookrightarrow]& 0
\end{tikzcd}
\end{equation}
with the top row given by the direct sum decomposition \eqref{reduction lemma direct decomposition with a semistable factor} and the middle vertical arrow being a minuscule effective modification at $\FFclosedpt$ (of degree $d$). 
\smallskip

\item There exists a
short exact sequence 
\[0 \longrightarrow \modif{\gensubvb} \longrightarrow \modif{\genextvb} \longrightarrow \modif{\genquotvb} \longrightarrow 0.\]
\end{enumerate}
Since $w$ is arbitrary, we deduce the desired assertion.
\end{proof}

\begin{remark}
The necessity part of Proposition \ref{reduction lemma for minuscule modifications} is evident as every minuscule effective modification $\modif{\genvb} \inj \genvb$ at $\FFclosedpt$ gives rise to a commutative diagram \eqref{extension diagram of minuscule modifications}. The main point of Proposition \ref{reduction lemma for minuscule modifications} is the sufficiency part, which is essentially equivalent to the identity \eqref{Chen-Tong identity} by Chen-Tong \cite{CT_weakaddandnewton}. 
\end{remark}

\begin{prop}[{\cite[\S5.5.2.1]{FF_curve}}]\label{additivity of degree for modifications}
Let $\genvb$ be a vector bundle on $\schff$. For every minuscule effective modification $\modif{\genvb} \inj \genvb$ at $\FFclosedpt$, its degree is equal to $\deg(\genvb) - \deg(\modif{\genvb})$. 
\end{prop}

\begin{lemma}\label{minuscule modification of a semistable bundle}
Let $\genvb$ and $\modif{\genvb}$ be vector bundles on $\schff$ of rank $n$ such that $\genvb$ is semistable. 
Take $\gencochar$ to be the minuscule dominant cocharacter of $\GL_n$ of degree $d:= \deg(\genvb) - \deg(\modif{\genvb})$ with slopes $0$ and $1$. 
There exists a minuscule effective modification $\modif{\genvb} \inj \genvb$ at $\FFclosedpt$ if and only if 
$\genvb$ and $\modif{\genvb}$ satisfy the following equivalent inequalities:
\begin{equation}\label{equiv inequalities for minuscule modification of a semistable bundle}
\HN(\genvb) + \dualpolygon{\gencochar} \geq \HN(\modif{\genvb}) \quad\text{ and } \quad
\HN(\modif{\genvb}) + \lineseg{1}{n} \slodom \HN(\genvb) \slodom \HN(\modif{\genvb}).
\end{equation}
\end{lemma}

\begin{proof}
By Proposition \ref{classification of nonempty newton strata with respect to basic element}, Proposition \ref{minuscule newton strata and minuscule modification} and Proposition \ref{additivity of degree for modifications}, there exists a minuscule effective modification $\modif{\genvb} \inj \genvb$ at $\FFclosedpt$ if and only if 
$\genvb$ and $\modif{\genvb}$ satisfy the first inequality in \eqref{equiv inequalities for minuscule modification of a semistable bundle}. 
If we write $\genslope$ for the slope of the line segment $\HN(\genvb)$, the polygon $\HN(\genvb) + \dualpolygon{\gencochar}$ has two distinct slopes $\genslope$ and $\genslope-1$. Hence it is not hard to verify the equivalence of 
the two inequalities in \eqref{equiv inequalities for minuscule modification of a semistable bundle}
by the concavity of HN polygons, 
thereby deducing the desired assertion. 
\end{proof}

\begin{theorem}\label{classification of nonempty minuscule Newton strata, general case}
Let $\gencochar$ be a minuscule dominant cocharacter of $\GL_n$ with slopes $0$ and $1$. 
Consider two arbitrary elements $b, \modif{b} \in B(\GL_n)$. 
Take a direct sum decomposition
\[\genvb_b \simeq \genvb_a \oplus \genvb_c \quad \text{ with } \quad a \in B(\GL_r) \text{ and } c \in B(\GL_{n-r})\]
such that $\HN(\genvb_a)$ coincides with the line segment of maximal slope in $\HN(\genvb_b)$.
%
\begin{enumerate}[label=(\arabic*)]
\item If the degree of $\gencochar$ is not equal to $\deg(\genvb_b) - \deg(\genvb_{\modif{b}})$, the Newton stratum $\Gr_{\GL_n, \gencochar, b}^{\modif{b}}$ is empty.
\smallskip

\item If the degree of $\gencochar$ is equal to $\deg(\genvb_b) - \deg(\genvb_{\modif{b}})$ the Newton stratum $\Gr_{\GL_n, \gencochar, b}^{\modif{b}}$ is nonempty if and only if there exist $\modif{a} \in B(\GL_{r})$ and $\modif{c} \in B(\GL_{n- r})$ with the following properties:
\smallskip
\begin{enumerate}[label=(\roman*)]
\item\label{slopewise dominance inequality for direct summand of max slope} We have $\HN(\genvb_{\modif{a}}) + \lineseg{1}{r} \slodom \HN(\genvb_{a}) \slodom \HN(\genvb_{\modif{a}})$.
\smallskip

\item\label{condition on HN polygons for extensions of modifications} If we write the HN decomposition of $\genvb_{\modif{c}}$ as 
\[ \genvb_{\modif{c}} \simeq \bigoplus_{i=1}^\numslope \genquotvb_i\]
where $\genquotvb_i$ are arranged in order of descending slope, there exists a sequence of vector bundles $\genvb_{\modif{a}} = \genvb_0, \genvb_1, \cdots, \genvb_\numslope = \genvb_b$ on $\schff$ such that $\HN(\genvb_{i-1} \oplus \genquotvb_i)$ has an $(\genextvb_{i-1}, \genextvb_i, \genquotvb_i)$-permutation for each $i = 1, \cdots, r$. 
\smallskip

\item\label{nonemptiness of reduced newton stratum} The Newton stratum $\Gr_{\GL_n, \genred{\gencochar}, c}^{\modif{c}}$ is nonempty where $\genred{\gencochar}$ is a minuscule dominant cocharacter of $\GL_{n-r}$ of degree $\genred{d}:= \deg(\genvb_c) - \deg(\genvb_{\modif{c}})$ with slopes $0$ and $1$. 
\end{enumerate}
\end{enumerate}
\end{theorem}

\begin{proof}
The assertion is straightforward to verify by Proposition \ref{classification of extensions}, Proposition \ref{minuscule newton strata and minuscule modification}, Proposition \ref{reduction lemma for minuscule modifications}, Proposition \ref{additivity of degree for modifications}, and Lemma \ref{minuscule modification of a semistable bundle}. 
\end{proof}

\begin{remark}
The elements $a \in B(\GL_r)$ and $c \in B(\GL_{n-r})$ are uniquely determined by the HN decomposition of $\genvb_b$. In addition, the Schubert cell $\Gr_{\GL_n, \genred{\gencochar}}$ contains finitely many nonemtpy Newton strata, as easily seen by Proposition \ref{subsheaves and slopewise dominance} and Proposition \ref{minuscule newton strata and minuscule modification}. Hence the conditions \ref{slopewise dominance inequality for direct summand of max slope} and \ref{nonemptiness of reduced newton stratum} together yield finitely many candidates for $\modif{a} \in B(\GL_{r})$ and $\modif{c} \in B(\GL_{n- r})$. We can thus use Theorem \ref{classification of nonempty minuscule Newton strata, general case} to inductively classify all nonempty Newton strata in an arbitrary minuscule Schubert cell of $\Gr_{\GL_n}$. 
\end{remark}

\subsection{Concave rationally tuplar polygons}\label{combinatorics on polygons}$ $

\begin{defn}
Given a rationally tuplar polygon $\genpolygon$,
%
we define its \emph{concave rearrangement} to be the rationally tuplar polygon $\concrearr{\genpolygon}$ such that the tuple $(\HNslope{\concrearr{\genpolygon}}{i})$ is the rearrangement of $(\HNslope{\genpolygon}{i})$ in descending order. 
\end{defn}

%

\begin{lemma}\label{maximality of concave rearrangement}
For every rationally tuplar polygon $\genpolygon$, we have $\concrearr{\genpolygon} \geq \genpolygon$. 
\end{lemma}

\begin{proof}
The assertion is evident by definition. 
\end{proof}

\begin{remark}
In fact, $\concrearr{\genpolygon}$ is the maximal rearrangement of $\genpolygon$ with respect to the Bruhat order. 
\end{remark}

\begin{defn}
Given two rationally tuplar polygon $\genpolygon$ and $\gensecpolygon$, 
we define their direct sum $\genpolygon \oplus \gensecpolygon$ to be the concave rearrangement of the concatenation of $\genpolygon$ and $\gensecpolygon$. 
\end{defn}

\begin{example}\label{direct sum of HN polygons}
Let us record some important examples of direct sums for our purpose. 
\begin{enumerate}[label=(\arabic*)]
\item For two vector bundles $\genvb$ and $\gensecvb$ on $\schff$, we have $\HN(\genvb \oplus \gensecvb) = \HN(\genvb) \oplus \HN(\gensecvb)$. 
\smallskip

\item For two minuscule dominant cocharacters $\gencochar_1$ of $\GL_{n_1}$ and $\gencochar_2$ of $\GL_{n_2}$ with slopes $0$ and $1$, their direct sum (as a rationally tuplar polygon) is a minuscule dominant cocharacter of $\GL_{n_1+n_2}$ with slopes $0$ and $1$. 
\end{enumerate}
\end{example}

\begin{lemma}\label{bruhat order and direct sums}
Given concave rationally tuplar polygons $\genpolygon, \modif{\genpolygon}, \gensecpolygon$ and $\modif{\gensecpolygon}$ with $\genpolygon \geq \modif{\genpolygon}$ and $\gensecpolygon \geq \modif{\gensecpolygon}$, we have $\genpolygon \oplus \gensecpolygon \geq \modif{\genpolygon} \oplus \modif{\gensecpolygon}$. 
\end{lemma}

\begin{proof}
Let $m$ and $n$ respectively denote the ranks of $\genpolygon$ and $\gensecpolygon$. Take two sets $A$ and $B$ which form a partition of the set $\{1, \cdots, m+n\}$ with
\[ (\HNslope{\modif{\genpolygon}\oplus \modif{\gensecpolygon}}{i})_{i \in A} = (\HNslope{\modif{\genpolygon}}{i}) \quad \text{ and } \quad (\HNslope{\modif{\genpolygon}\oplus \modif{\gensecpolygon}}{i})_{i \in B} = (\HNslope{\modif{\gensecpolygon}}{i}).\]
Let $\genthirdpolygon$ to be the rationally tuplar polygon of rank $m+n$ with
\[ (\HNslope{\genthirdpolygon}{i})_{i \in A} = (\HNslope{\genpolygon}{i}) \quad \text{ and } \quad (\HNslope{\genthirdpolygon}{i})_{i \in B} = (\HNslope{\gensecpolygon}{i}).\]
Since $\genpolygon, \modif{\genpolygon}, \gensecpolygon$ and $\modif{\gensecpolygon}$ are all concave, the inequalities $\genpolygon \geq \modif{\genpolygon}$ and $\gensecpolygon \geq \modif{\gensecpolygon}$ together imply $\genthirdpolygon \geq \modif{\genpolygon} \oplus \modif{\gensecpolygon}$. 
Now we find $\genpolygon \oplus \gensecpolygon = \concrearr{\genthirdpolygon} \geq \genthirdpolygon$ by Lemma \ref{maximality of concave rearrangement} to complete the proof. 
\end{proof}

\begin{remark}
Lemma \ref{bruhat order and direct sums} does not hold without the concavity assumption. For example, if we take $\genpolygon = \gensecpolygon = \lineseg{d/r}{r}$ for some integers $r$ and $d$ with $r>0$, for arbitrary nonlinear convex polygons $\modif{\genpolygon}$ and $\modif{\gensecpolygon}$ of rank $r$ and degree $d$ we do not have $\genpolygon \oplus \gensecpolygon \geq \modif{\genpolygon} \oplus \modif{\gensecpolygon}$ despite having $\genpolygon \geq \modif{\genpolygon}$ and $\gensecpolygon \geq \modif{\gensecpolygon}$, as illustrated in Figure \ref{counter example for direct sum inequality without concavity assumption}. 
\begin{figure}[H]
\begin{tikzpicture}[scale=1]	
		\coordinate (left) at (0, 0);
		\coordinate (p1) at (2, 1);
		\coordinate (right) at (4, 2);

		\coordinate (q1) at (1.5, 0);
		\coordinate (q2) at (3.5, 1);
		
		\coordinate (r1) at (1, 2);
				
		\draw[step=1cm,thick, color=red] (left) -- (p1);
		\draw[step=1cm,thick, color=orange] (p1) -- (right);
		\draw[step=1cm,thick, color=blue] (left) -- (q1) -- (p1);
		\draw[step=1cm,thick, color=violet] (p1) -- (q2) -- (right);
		
		\draw[step=1cm,thick, color=teal] (left) -- (r1) -- (right);

		\draw [fill] (q1) circle [radius=0.05];		
		\draw [fill] (q2) circle [radius=0.05];		
		\draw [fill] (r1) circle [radius=0.05];			
		\draw [fill] (left) circle [radius=0.05];
		\draw [fill] (right) circle [radius=0.05];		
		
		\draw [fill] (p1) circle [radius=0.05];		

		\path (left) ++(1, 0.7) node {\color{red}$\genpolygon$};
		\path (p1) ++(0.8, 0.7) node {\color{orange}$\gensecpolygon$};		
		\path (q1) ++(0.3, 0) node {\color{blue}$\modif{\genpolygon}$};
		\path (q2) ++(0.3, 0) node {\color{violet}$\modif{\gensecpolygon}$};
		\path (r1) ++(-1, -0.2) node {\color{teal}$\modif{\genpolygon} \oplus \modif{\gensecpolygon}$};
\end{tikzpicture}
\caption{A counter example for Lemma \ref{bruhat order and direct sums} without the concavity assumption}\label{counter example for direct sum inequality without concavity assumption}
\end{figure}
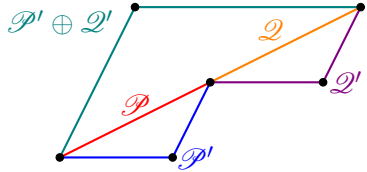
\end{remark}

\begin{lemma}\label{inequality on direct sum and sum}
Let $\genpolygon$ and $\gensecpolygon$ be rationally tuplar polygons of rank $m$ and $n$. For arbitrary rationally tuplar polygons $\modif{\genpolygon}$ of rank $m$ and $\modif{\gensecpolygon}$ of rank $n$, we have
\[ (\genpolygon \oplus \gensecpolygon) + (\modif{\genpolygon} \oplus \modif{\gensecpolygon}) \geq (\genpolygon + \modif{\genpolygon}) \oplus (\gensecpolygon + \modif{\gensecpolygon}).\]
\end{lemma}

\begin{proof}
We observe that there exist permutations $\genperm$ and $\modif{\genperm}$ of the set $\{1, \cdots, m+n\}$ with 
\[ \HNslope{(\genpolygon + \modif{\genpolygon}) \oplus (\gensecpolygon + \modif{\gensecpolygon})}{i} = \HNslope{\genpolygon \oplus \gensecpolygon}{\genperm(i)} + \HNslope{\modif{\genpolygon} \oplus \modif{\gensecpolygon}}{\modif{\genperm}(i)} \quad \text{ for each } i = 1, \cdots, m+n,\]
and consequently deduce the desired assertion by the concavity of $\genpolygon \oplus \gensecpolygon$ and $\modif{\genpolygon} \oplus \modif{\gensecpolygon}$. 
\end{proof}

\subsection{An explicit classification of nonempty Newton strata }\label{subsection for explicit classification of Newton strata}$ $

\begin{lemma}\label{symmetry of minuscule modifications}
Let $\genvb$ be a vector bundle on $\schff$ of rank $n$. Every minuscule effective modification $\modif{\genvb} \inj \genvb$ at $\FFclosedpt$ gives rise to a minuscule effective modification $\genlift{\genvb} \inj \modif{\genvb}$ at $\FFclosedpt$ with $\HN(\genlift{\genvb}) = \HN(\genvb) - \lineseg{1}{n}$.
\end{lemma}

\begin{proof}
Let $\gencochar$ be the minuscule dominant cocharacter of $\GL_n$ of degree $d: = \deg(\genvb) - \deg(\modif{\genvb})$ with slopes $0$ and $1$. Take elements $b, \modif{b}$ and $\genlift{b}$ in $B(\GL_n)$ with $\genvb \simeq \genvb_b$, $\modif{\genvb} \simeq \genvb_{\modif{b}}$ and $\genlift{\genvb} \simeq \genvb_{\genlift{b}}$. The effective modification $\modif{\genvb} \inj \genvb$ at $\FFclosedpt$ yields a point in $\Gr_{\GL_n, \gencochar, b}^{\modif{b}}$ by Proposition \ref{minuscule newton strata and minuscule modification} and Proposition \ref{additivity of degree for modifications}, and in turn yields a point in $\Gr_{\GL_n, \dualpolygon{\gencochar}+\lineseg{1}{n}, \modif{b}}^{\genlift{b}}$ by Proposition \ref{reduction to minimal minuscule}. Hence we obtain a minuscule effective modification $\genlift{\genvb} \inj \modif{\genvb}$ at $\FFclosedpt$ by Proposition \ref{minuscule newton strata and minuscule modification} as desired. 
\end{proof}

\begin{prop}\label{classification of minuscule modifications, necessity part}
Let $\genvb$ be a vector bundle on $\schff$ of rank $n$. For every minuscule effective modification $\modif{\genvb} \inj \genvb$ at $\FFclosedpt$, we have 
\[\HN(\genvb) + \dualpolygon{\gencochar} \geq \HN(\modif{\genvb}) \quad\text{ and } \quad \HN(\modif{\genvb}) + \lineseg{1}{n} \slodom \HN(\genvb) \slodom \HN(\modif{\genvb})\]
where $\gencochar$ is the minuscule dominant cocharacter of $\GL_n$ of degree $d:= \deg(\genvb) - \deg(\modif{\genvb})$ with slopes $0$ and $1$. 
\end{prop}

\begin{proof}
The second inequality is an immediate consequence of Proposition \ref{subsheaves and slopewise dominance} and Lemma \ref{symmetry of minuscule modifications}. Hence it remains to establish the first inequality. Let us write $\numslope$ for the number of distinct slopes in $\HN(\genvb)$ and proceed by induction on $\numslope$. 
If $\genvb$ is semistable, the assertion is evident by Lemma \ref{minuscule modification of a semistable bundle}. We henceforth assume that $\genvb$ is not semistable, so that we have $\numslope > 1$. Take a direct sum decomposition
\begin{equation*}
\genvb \simeq \gensubvb \oplus \genquotvb
\end{equation*}
such that $\HN(\gensubvb)$ coincides with the line segment of maximal slope in $\HN(\genvb)$. 
The numbers of distinct slopes in $\HN(\gensubvb)$ and $\HN(\genquotvb)$ are respectively $1$ and $\numslope -1$. 
Now Proposition \ref{reduction lemma for minuscule modifications} yields minuscule effective modifications $\alpha: \modif{\gensubvb} \inj \gensubvb$ and $\beta: \modif{\genquotvb} \inj \genquotvb$ at $\FFclosedpt$ with a short exact sequence
\begin{equation}\label{minuscule modification short exact sequence, necessity part}
0 \longrightarrow \modif{\gensubvb} \longrightarrow \modif{\genextvb} \longrightarrow \modif{\genquotvb} \longrightarrow 0.
\end{equation}
Let us denote the types of $\alpha$ and $\beta$ respectively by $\gencochar_1$ and $\gencochar_2$. 
In a concrete form, we have 
\[\gencochar_1 = \lineseg{1}{d_1} \oplus \lineseg{0}{n_1 - d_1} \quad \text{ and } \quad \gencochar_2 = \lineseg{1}{d_2} \oplus \lineseg{0}{n_2 - d_2}\]
where we set $n_1: = \rk(\gensubvb) = \rk(\modif{\gensubvb})$, $n_2: = \rk(\genquotvb) = \rk(\modif{\genquotvb})$, $d_1:= \deg(\gensubvb) - \deg(\modif{\gensubvb})$ and $d_2:= \deg(\genquotvb) - \deg(\modif{\genquotvb})$. 
By the induction hypothesis, the minuscule effective modifications $\alpha$ and $\beta$ at $\FFclosedpt$ respectively yield the inequalities
\[ \HN(\gensubvb) + \dualpolygon{\gencochar}_1 \geq \HN(\modif{\gensubvb}) \quad \text{ and } \quad \HN(\genquotvb) + \dualpolygon{\gencochar}_2 \geq \HN(\modif{\genquotvb}).\]
Then by Example \ref{direct sum of HN polygons}, Lemma \ref{bruhat order and direct sums} and Lemma \ref{inequality on direct sum and sum} we find
\[ \HN(\genvb) + \dualpolygon{\gencochar} = (\HN(\gensubvb) \oplus \HN(\genquotvb)) + (\dualpolygon{\gencochar}_1 \oplus \dualpolygon{\gencochar}_2) \geq (\HN(\gensubvb) + \dualpolygon{\gencochar}_1) \oplus (\HN(\genquotvb) + \dualpolygon{\gencochar}_2)  \geq \HN(\modif{\gensubvb} \oplus \modif{\genquotvb}).\]
In addition, by Proposition \ref{classification of extensions, necessity condition on E-permutation}
the short exact sequence \eqref{minuscule modification short exact sequence, necessity part} yields the inequality
\[ \HN(\modif{\gensubvb} \oplus \modif{\genquotvb}) \geq \HN(\modif{\genvb}).\]
We thus obtain the first inequality, thereby completing the proof. 
\end{proof}

\begin{remark}
The two inequalities in Proposition \ref{classification of minuscule modifications, necessity part} are not equivalent in general, although they are equivalent if $\genvb$ is semistable as shown in Lemma \ref{minuscule modification of a semistable bundle}. 
\end{remark}

\begin{example}\label{counterexample for converse of minuscule inequalities}
Let us present an example showing that the converse of Proposition \ref{classification of minuscule modifications, necessity part} does not hold. Take $\genvb$ and $\modif{\genvb}$ to be vector bundles on $\schff$ with
\[ \HN(\genvb) = \lineseg{4/3}{3} \oplus \lineseg{3/4}{4} \quad \text{ and } \quad \HN(\modif{\genvb}) = \lineseg{1}{2} \oplus \lineseg{1/3}{3} \oplus \lineseg{0}{2}.\]
By construction, we have $\rk(\genvb) = \rk(\modif{\genvb}) = 7$, $\deg(\genvb) = 7$ and $\deg(\modif{\genvb}) = 3$. 
Now for the minuscule dominant cocharacter $\gencochar$ of $\GL_7$ of degree $4$ with slopes $0$ and $1$, we find
\[\HN(\genvb) + \dualpolygon{\gencochar} \geq \HN(\modif{\genvb}) \quad\text{ and } \quad \HN(\modif{\genvb}) + \lineseg{1}{7} \slodom \HN(\genvb) \slodom \HN(\modif{\genvb}).\]
\begin{figure}[H]
\begin{tikzpicture}[scale=0.75]
		\draw[step=1, dotted, gray, very thin] (-0.3,-0.3) grid (7.3,10.3);
		
		\coordinate (p00) at (2, 4);
		\coordinate (p11) at (5, 8);
		\coordinate (p22) at (7, 10);

		\coordinate (left) at (0, 0);
		\coordinate (q0) at (3, 4);
		\coordinate (q1) at (7, 7);
		
		\coordinate (r1) at (q0);

		\coordinate (p0) at (2, 2);
		\coordinate (p1) at (5, 3);
		\coordinate (p2) at (7, 3);
				
		\draw[step=1cm,thick, color=red] (left) -- (q0) --  (q1);
		\draw[step=1cm,thick, color=mynicegreen] (left) -- (p0) --  (p1) -- (p2);
		\draw[step=1cm,thick, color=teal] (left) -- (p00) --  (p11) -- (p22);
		\draw[step=1cm,thick, color=blue] (r1) -- (p2);
		
		\draw [fill] (q0) circle [radius=0.05];		
		\draw [fill] (q1) circle [radius=0.05];		
		\draw [fill] (left) circle [radius=0.05];
		
		\draw [fill] (r1) circle [radius=0.05];		
		
		\draw [fill] (p0) circle [radius=0.05];		
		\draw [fill] (p1) circle [radius=0.05];		
		\draw [fill] (p2) circle [radius=0.05];		

		\draw [fill] (p00) circle [radius=0.05];		
		\draw [fill] (p11) circle [radius=0.05];		
		\draw [fill] (p22) circle [radius=0.05];	
		

		
		\path (q1) ++(1.0, 0) node {\color{red}$\HN(\genvb)$};
		\path (r1) ++(2.9, -0.1) node {\color{blue}$\HN(\genvb) + \dualpolygon{\gencochar}$};
		\path (p00) ++(-1.6, 0) node {\color{teal}$\HN(\modif{\genvb}) + \lineseg{1}{7}$};
		\path (p0) ++(0.7, -0.5) node {\color{mynicegreen}$\HN(\modif{\genvb})$};

\end{tikzpicture}
\caption{A counter example for the converse of Proposition \ref{classification of minuscule modifications, necessity part}}\label{counter example for converse of HN inequalities}
\end{figure}
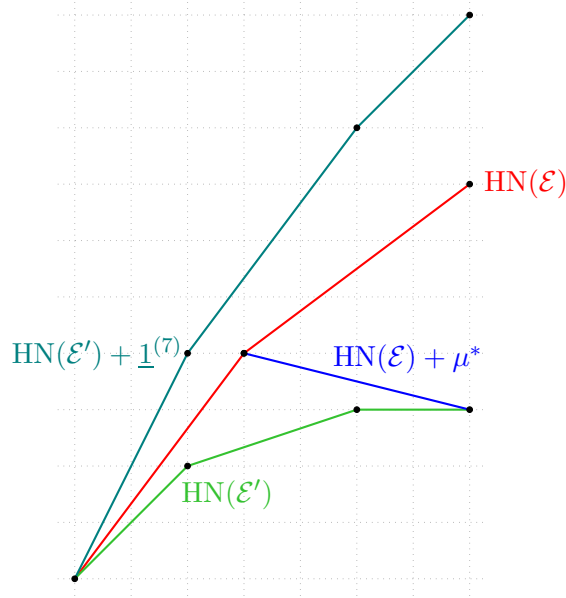
We wish to show that there does not exist a minuscule effective modification $\modif{\genvb} \inj \genvb$ at $\FFclosedpt$. Suppose for contradiction that such a modification exists. Take a direct sum decomposition 
\[\genvb \simeq \gensubvb \oplus \genquotvb\]
with $\HN(\gensubvb) = \lineseg{4/3}{3}$ and $\HN(\genquotvb) = \lineseg{3/4}{4}$. Proposition \ref{reduction lemma for minuscule modifications} yields minuscule effective modifications $\modif{\gensubvb} \inj \gensubvb$ and $\modif{\genquotvb} \inj \genquotvb$ at $\FFclosedpt$ with a short exact sequence
\begin{equation*}\label{minuscule modification short exact sequence for converse counter example}
0 \longrightarrow \modif{\gensubvb} \longrightarrow \modif{\genextvb} \longrightarrow \modif{\genquotvb} \longrightarrow 0.
\end{equation*}
Then by Proposition \ref{classification of extensions, necessity condition on E-permutation} we obtain a $(\modif{\gensubvb}, \modif{\genextvb}, \modif{\genquotvb})$-permutation $\genpolygon$ of $\HN(\modif{\gensubvb} \oplus \modif{\genquotvb})$. Since we have $\genpolygon \geq \HN(\modif{\genvb})$ by construction, we find
\begin{equation}\label{lower bounds for E'-permutation, converse counter example}
\HNslope{\genpolygon}{1} \geq \HNslope{\HN(\modif{\genvb})}{1} = 1 \quad\text{ and } \quad \HNslope{\genpolygon}{2} \geq \HNslope{\HN(\modif{\genvb})}{2} = 1.
\end{equation}
Moreover, as $\modif{\genquotvb}$ is a subsheaf of $\genquotvb$ by construction, Proposition \ref{subsheaves and slopewise dominance} implies that all slopes in $\HN(\modif{\genquotvb})$ are less than or equal to $3/4$. 
We then deduce by \eqref{lower bounds for E'-permutation, converse counter example} that $\HNslope{\genpolygon}{1}$ and $\HNslope{\genpolygon}{2}$ should occur as a slope of $\modif{\gensubvb}$, and in turn find that the inequalities in \eqref{lower bounds for E'-permutation, converse counter example} are in fact equalities. Therefore $\HN(\modif{\gensubvb})$ must contain the line segment $\lineseg{1}{2}$, and consequently is given by $ \lineseg{1}{2} \oplus \lineseg{d}{1}$ for some integer $d$. Then we have $d = \HNslope{\genpolygon}{i}$ for some $i >2$ and thus find $d \leq \HNslope{\HN(\modif{\genvb})}{i} \leq 1/3$. On the other hand, since $\modif{\gensubvb}$ occurs as a minuscule effective modification of $\gensubvb$ at $\FFresfield$, Proposition \ref{classification of minuscule modifications, necessity part} implies $d \geq 1/3$. Now we have a desired contradiction as $d$ is an integer with $d \leq 1/3$ and $d \geq 1/3$. 
\end{example}

\begin{prop}\label{classification of minuscule modifications, case of sufficiently distinct slopes}
Let $\genvb$ and $\modif{\genvb}$ be vector bundles on $\schff$ of rank $n$.
Denote by $\gencochar$ the minuscule dominant cocharacter of $\GL_n$ of degree $d:= \deg(\genvb) - \deg(\modif{\genvb})$ with slopes $0$ and $1$. Assume that $\genvb$ satisfies the following property:
\begin{enumerate}[label=($\ast$)]
\item\label{slope condition for noninductive criterion} All distinct slopes in $\HN(\genvb)$ differ by more than $1$. 
\end{enumerate}
There exists a minuscule effective modification $\modif{\genvb} \inj \genvb$ at $\FFclosedpt$ if and only if  
$\genvb$ and $\modif{\genvb}$ satisfy the following conditions:
\begin{enumerate}[label=(\roman*)]
\item\label{mazur inequality and slopewise dominance for minuscule modifications} We have $\HN(\genvb) + \dualpolygon{\gencochar} \geq \HN(\modif{\genvb})$ and $\HN(\modif{\genvb}) + \lineseg{1}{n} \slodom \HN(\genvb) \slodom \HN(\modif{\genvb})$. 
\smallskip

\item\label{common breakpoints for minuscule modifications} For each breakpoint of $\HN(\genvb)$, there exists 
a breakpoint of $\HN(\modif{\genvb})$ with the same $x$-coordinate. 
\end{enumerate}
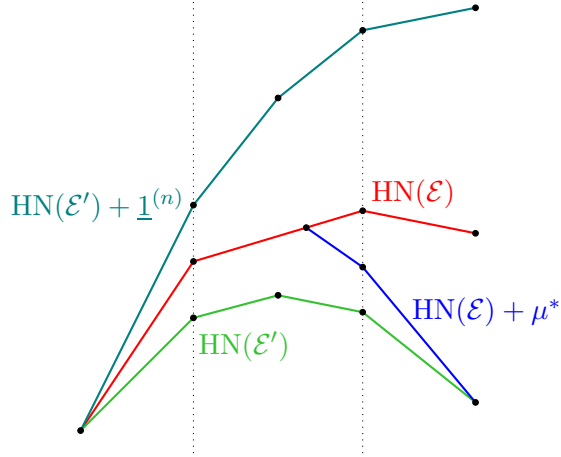
\begin{figure}[H]
\begin{tikzpicture}[scale=0.75]
		\coordinate (p00) at (2, 4);
		\coordinate (p11) at (3.5, 5.9);
		\coordinate (p22) at (5, 7.1);
		\coordinate (p33) at (7, 7.5);

		\coordinate (left) at (0, 0);
		\coordinate (q0) at (2, 3);
		\coordinate (q1) at (5, 3.9);
		\coordinate (q2) at (7, 3.5);
		
		\coordinate (r1) at (4, 3.6);
		\coordinate (r2) at (5, 2.9);

		\coordinate (p0) at (2, 2);
		\coordinate (p1) at (3.5, 2.4);
		\coordinate (p2) at (5, 2.1);
		\coordinate (p3) at (7, 0.5);
				
		\draw[step=1cm,thick, color=red] (left) -- (q0) --  (q1) -- (q2);
		\draw[step=1cm,thick, color=mynicegreen] (left) -- (p0) --  (p1) -- (p2) -- (p3);
		\draw[step=1cm,thick, color=teal] (left) -- (p00) --  (p11) -- (p22) -- (p33);
		\draw[step=1cm,thick, color=blue] (r1) -- (r2) -- (p3);
		
		\draw [fill] (q0) circle [radius=0.05];		
		\draw [fill] (q1) circle [radius=0.05];		
		\draw [fill] (q2) circle [radius=0.05];		
		\draw [fill] (left) circle [radius=0.05];
		
		\draw [fill] (r1) circle [radius=0.05];		
		\draw [fill] (r2) circle [radius=0.05];	
		
		\draw [fill] (p0) circle [radius=0.05];		
		\draw [fill] (p1) circle [radius=0.05];		
		\draw [fill] (p2) circle [radius=0.05];		
		\draw [fill] (p3) circle [radius=0.05];		

		\draw [fill] (p00) circle [radius=0.05];		
		\draw [fill] (p11) circle [radius=0.05];		
		\draw [fill] (p22) circle [radius=0.05];	
		\draw [fill] (p33) circle [radius=0.05];	
		
		\draw[step=1cm,dotted] (2, -0.4) -- (2, 7.7);
		\draw[step=1cm,dotted] (5, -0.4) -- (5, 7.7);

		
		\path (q1) ++(0.9, 0.3) node {\color{red}$\HN(\genvb)$};
		\path (r2) ++(2.2, -0.8) node {\color{blue}$\HN(\genvb) + \dualpolygon{\gencochar}$};
		\path (p00) ++(-1.7, 0) node {\color{teal}$\HN(\modif{\genvb}) + \lineseg{1}{n}$};
		\path (p0) ++(0.9, -0.5) node {\color{mynicegreen}$\HN(\modif{\genvb})$};

\end{tikzpicture}
\caption{Illustration of the conditions in Proposition \ref{classification of minuscule modifications, case of sufficiently distinct slopes}}
\end{figure}
\end{prop}

\begin{proof}
Let us first assume that $\genvb$ and $\modif{\genvb}$ satisfy the conditions \ref{mazur inequality and slopewise dominance for minuscule modifications} and \ref{common breakpoints for minuscule modifications}. We write the HN decomposition of $\genvb$ as
\begin{equation}\label{HN decomposition for classification of minuscule modifications, case of sufficiently distinct slopes}
\genvb \simeq \bigoplus_{i = 1}^\numslope \genvb_i
\end{equation}
where the direct summands $\genvb_i$ are arranged in order of descending slope, and set
\[ x_i := \sum_{j = 1}^i \rk(\genvb_j) \quad \text{ for } i = 0, \cdots, \numslope.\]
By the condition \ref{common breakpoints for minuscule modifications}, we get a direct sum decomposition
\begin{equation}\label{common break point decomposition}
\modif{\genvb} \simeq  \bigoplus_{i = 1}^\numslope \modif{\genvb}_i
\end{equation}
where each $\HN(\modif{\genvb}_i)$ coincides with the restriction of $\HN(\modif{\genvb})$ on the interval $[x_{i-1}, x_i]$. Then by the condition \ref{mazur inequality and slopewise dominance for minuscule modifications} we find
\[ \HN(\modif{\genvb}_i) + \lineseg{1}{x_i - x_{i-1}} \slodom \HN(\genvb_i) \slodom \HN(\modif{\genvb}_i) \quad \text{ for } i = 1, \cdots, \numslope.\]
Now for each $i = 1, \cdots, \numslope$, Lemma \ref{minuscule modification of a semistable bundle} yields a minuscule effective modification $\modif{\genvb}_i \inj \genvb_i$ at $\FFclosedpt$ as $\genvb_i$ is semistable. Hence we obtain a minuscule effective modification $\modif{\genvb} \inj \genvb$ at $\FFclosedpt$ from the direct sum decompositions \eqref{HN decomposition for classification of minuscule modifications, case of sufficiently distinct slopes} and \eqref{common break point decomposition}.

For the converse, we now assume that there exists a minuscule effective modification $\modif{\genvb} \inj \genvb$ at $\FFclosedpt$. Since $\genvb$ and $\modif{\genvb}$ satisfy the condition \ref{mazur inequality and slopewise dominance for minuscule modifications} by Proposition \ref{classification of minuscule modifications, necessity part}, it remains to establish the condition \ref{common breakpoints for minuscule modifications}. We proceed by induction on the number $\numslope$ of distinct slopes in $\HN(\genvb)$. If $\genvb$ is semistable, the assertion 
is vacuously true as $\HN(\genvb)$ does not have a breakpoint. We henceforth assume that $\genvb$ is not semistable, so that we have $\numslope > 1$. Take a direct sum decomposition
\begin{equation}\label{max slope factor decomp, case of sufficiently distinct slopes}
\genvb \simeq \gensubvb \oplus \genquotvb
\end{equation}
such that $\HN(\gensubvb)$ coincides with the line segment of maximal slope in $\HN(\genvb)$. Let us denote the slope of $\HN(\gensubvb)$ by $\genslope$. By construction, $\HN(\genquotvb)$ has $\numslope-1$ distinct slopes 
which are all less than $\genslope -1$
by the property \ref{slope condition for noninductive criterion}. In addition, we have $\HN(\modif{\genvb}) + \lineseg{1}{n} \slodom \HN(\genvb)$ by Proposition \ref{classification of minuscule modifications, necessity part} and thus find
\begin{equation}\label{minuscule modification slope lower bound up to first common break point} 
\HNslope{\HN(\modif{\genvb})}{i} \geq \genslope-1 \quad \text{ for } i = 1, \cdots, \rk(\gensubvb). 
\end{equation}
Now we note by Proposition \ref{reduction lemma for minuscule modifications} that there exist minuscule effective modifications $\modif{\gensubvb} \inj \gensubvb$ and $\modif{\genquotvb} \inj \genquotvb$ at $\FFclosedpt$ with a short exact sequence
\begin{equation}\label{minuscule modification short exact sequence, case of sufficiently distinct slopes}
0 \longrightarrow \modif{\gensubvb} \longrightarrow \modif{\genextvb} \longrightarrow \modif{\genquotvb} \longrightarrow 0.
\end{equation}
Then we find
\begin{equation}\label{minuscule modification of max slope factor quotient slope lower bound}
\HNslope{\HN(\modif{\genquotvb})}{i} \leq \HNslope{\HN(\genquotvb)}{i} < \genslope-1 \quad \text{ for } i = 1, \cdots, \rk(\modif{\genquotvb})
\end{equation}
by Proposition \ref{subsheaves and slopewise dominance}, and also obtain a $(\modif{\gensubvb}, \modif{\genextvb}, \modif{\genquotvb})$-permutation $\genpolygon$ of $\HN(\modif{\gensubvb} \oplus \modif{\genquotvb})$ by Proposition \ref{classification of extensions, necessity condition on E-permutation}. For each $i = 1, \cdots, \rk(\gensubvb)$, the inequalities \eqref{minuscule modification slope lower bound up to first common break point} and \eqref{minuscule modification of max slope factor quotient slope lower bound} together imply that $\HNslope{\genpolygon}{i}$ occurs as a slope in $\HN(\modif{\gensubvb})$.
Since we have $\genpolygon \geq \HN(\modif{\genvb})$ by construction, we find
\[ \HNslope{\genpolygon}{i} = \HNslope{\HN(\modif{\gensubvb})}{i} = \HNslope{\HN(\modif{\genvb})}{i} \quad \text{ for } i = 1 , \cdots, \rk(\gensubvb)\]
and consequently deduce from the inequalities \eqref{minuscule modification slope lower bound up to first common break point} and \eqref{minuscule modification of max slope factor quotient slope lower bound} that all slopes in $\HN(\modif{\gensubvb})$ are greater than all slopes in $\HN(\modif{\genquotvb})$. Hence the short exact sequence \eqref{minuscule modification short exact sequence, case of sufficiently distinct slopes} induces a direct sum
\begin{equation}\label{common breakpoint decomposition for minuscule modification}
\modif{\genvb} \simeq \modif{\gensubvb} \oplus \modif{\genquotvb}
\end{equation}
by Proposition \ref{splitting extension for vector bundles with dominating slopes}, and consequently yields a breakpoint of $\HN(\modif{\genvb})$ with $x$-coordinate $\rk(\modif{\gensubvb}) = \rk(\gensubvb)$. 
In addition, since we have a minuscule effective modification $\modif{\genquotvb} \inj \genquotvb$ at $\FFclosedpt$, we find by the induction hypothesis that for every breakpoint of $\HN(\genquotvb)$ there exists a breakpoint of $\HN(\modif{\genquotvb})$ with the same $x$-coordinate. We thus establish the condition \ref{common breakpoints for minuscule modifications} by the direct sum decompositions \eqref{max slope factor decomp, case of sufficiently distinct slopes} and \eqref{common breakpoint decomposition for minuscule modification}, thereby completing the proof.
\end{proof}

\begin{theorem}\label{classification of nonempty Newton strata, case of sufficiently distinct slopes}
Let $\gencochar$ be a minuscule dominant cocharacter of $\GL_n$ with slopes $0$ and $1$. 
Take two arbitrary elements $b, \modif{b} \in B(\GL_n)$ and write $\newtonmap{b}:= \HN(\genvb_b)$ and $\newtonmap{\modif{b}}:= \HN(\genvb_{\modif{b}})$.
Assume that $b$ satisfies the following property:
\begin{enumerate}[label=($\ast$)]
\item All distinct slopes in $\newtonmap{b}$ differ by more than $1$. 
\end{enumerate}
The Newton stratum $\Gr_{\GL_n, \gencochar, b}^{\modif{b}}$ is nonempty if and only if 
$\newtonmap{b}$ and $\newtonmap{\modif{b}}$ satisfy the following conditions:
\begin{enumerate}[label=(\roman*)]
\item\label{mazur inequality and slopewise dominance for nonempty newton strata} We have $\newtonmap{b} + \dualpolygon{\gencochar} \geq \newtonmap{\modif{b}}$ and $\newtonmap{\modif{b}} + \lineseg{1}{n} \slodom \newtonmap{b} \slodom \newtonmap{\modif{b}}$. 
\smallskip

\item\label{common breakpoint condition for nonempty newton strata, case of sufficiently distinct slopes} For each breakpoint of $\newtonmap{b}$, there exists 
a breakpoint of $\newtonmap{\modif{b}}$ with the same $x$-coordinate. 
\end{enumerate}
\end{theorem}

\begin{proof}
The assertion is an immediate consequence of Proposition \ref{minuscule newton strata and minuscule modification}, Proposition \ref{additivity of degree for modifications} and Proposition \ref{classification of minuscule modifications, case of sufficiently distinct slopes}.
\end{proof}

\begin{remark}
Theorem \ref{classification of nonempty Newton strata, case of sufficiently distinct slopes} is identical to Theorem \ref{classification of nonempty Newton strata, case of sufficiently distinct slopes intro}. For a non-minuscule cocharacter $\gencochar$ of $\GL_n$ with slopes in $[0, d]$, we should be able to get a similar classification theorem with $d$ in place of $1$ using the Demazure resolution. 
\end{remark}

\begin{example}\label{example of minuscule modification without common breakpoints}
Let us provide an example to show that Proposition \ref{classification of minuscule modifications, case of sufficiently distinct slopes} and Theorem \ref{classification of nonempty Newton strata, case of sufficiently distinct slopes} do not hold without assuming \ref{slope condition for noninductive criterion}. Take $\genvb$ and $\modif{\genvb}$ to be vector bundles on $\schff$ with
\[ \HN(\genvb) = \lineseg{5/4}{4} \oplus \lineseg{3/4}{4} \quad \text{ and } \quad \HN(\modif{\genvb}) = \lineseg{3/5}{5} \oplus \lineseg{1/3}{3}.\]
Then $\HN(\genvb)$ and $\HN(\modif{\genvb})$ do not have breakpoints with the same $x$-coordinates. 
We wish to show that there exists a minuscule effective modification $\modif{\genvb} \inj \genvb$ at $\FFclosedpt$. Take vector bundles $\gensubvb$, $\modif{\gensubvb}$, $\genquotvb$ and $\modif{\genquotvb}$ on $\schff$ with
\[ \HN(\gensubvb) = \lineseg{5/4}{4}, \quad \HN(\modif{\gensubvb}) = \lineseg{1/4}{4}, \quad \HN(\genquotvb) = \HN(\modif{\genquotvb}) = \lineseg{3/4}{4}.\]
By construction, we have a direct sum decomposition 
\[\genvb \simeq \gensubvb \oplus \genquotvb.\]
In addition, we obtain minuscule effective modifications $\modif{\gensubvb} \inj \gensubvb$ and $\modif{\genquotvb} \inj \genquotvb$ at $\FFclosedpt$ by Lemma \ref{minuscule modification of a semistable bundle}, and find a short exact sequence
\[0 \longrightarrow \modif{\gensubvb} \longrightarrow \modif{\genextvb} \longrightarrow \modif{\genquotvb} \longrightarrow 0\]
by Proposition \ref{classification of extensions}. Therefore Proposition \ref{reduction lemma for minuscule modifications} yields a minuscule effective modification $\modif{\genvb} \inj \genvb$ at $\FFclosedpt$ as desired. 

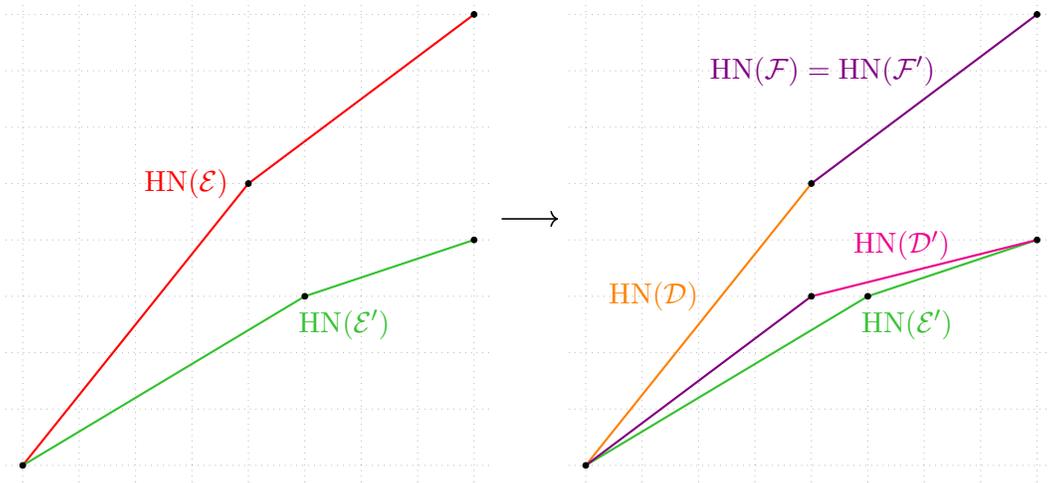
\begin{figure}[H]
\begin{tikzpicture}[scale=0.75]
		\draw[step=1, dotted, gray, very thin] (-0.3,-0.3) grid (8.3,8.3);
		

		\coordinate (left) at (0, 0);
		\coordinate (q0) at (4, 5);
		\coordinate (q1) at (8, 8);
		

		\coordinate (p0) at (5, 3);
		\coordinate (p1) at (8, 4);
				
		\draw[step=1cm,thick, color=red] (left) -- (q0) --  (q1);
		\draw[step=1cm,thick, color=mynicegreen] (left) -- (p0) --  (p1);
		
		\draw [fill] (q0) circle [radius=0.05];		
		\draw [fill] (q1) circle [radius=0.05];		
		\draw [fill] (left) circle [radius=0.05];
		
		
		\draw [fill] (p0) circle [radius=0.05];		
		\draw [fill] (p1) circle [radius=0.05];		

		

		
		\path (q0) ++(-1.1, 0) node {\color{red}$\HN(\genvb)$};
		\path (p0) ++(0.7, -0.5) node {\color{mynicegreen}$\HN(\modif{\genvb})$};

\end{tikzpicture}
\begin{tikzpicture}[scale=0.5]
        \pgfmathsetmacro{\textycoordinate}{7}
		\draw[->, line width=0.6pt] (0, \textycoordinate) -- (1.5,\textycoordinate);
		\draw (0,0) circle [radius=0.00];	
\end{tikzpicture}
\begin{tikzpicture}[scale=0.75]
		\draw[step=1, dotted, gray, very thin] (-0.3,-0.3) grid (8.3,8.3);
		

		\coordinate (left) at (0, 0);
		\coordinate (q0) at (4, 5);
		\coordinate (q1) at (8, 8);
		
		\coordinate (r1) at (4, 3);

		\coordinate (p0) at (5, 3);
		\coordinate (p1) at (8, 4);
				
		\draw[step=1cm,thick, color=orange] (left) -- (q0);
		\draw[step=1cm,thick, color=violet] (q0) -- (q1);
		\draw[step=1cm,thick, color=mynicegreen] (left) -- (p0) --  (p1);
		\draw[step=1cm,thick, color=violet] (left) -- (r1);
		\draw[step=1cm,thick, color=magenta] (r1) -- (p1);
		
		\draw [fill] (q0) circle [radius=0.05];		
		\draw [fill] (q1) circle [radius=0.05];		
		\draw [fill] (left) circle [radius=0.05];
		
		\draw [fill] (r1) circle [radius=0.05];		
		
		\draw [fill] (p0) circle [radius=0.05];		
		\draw [fill] (p1) circle [radius=0.05];		

		

		
		\path (q0) ++(-2.8, -2) node {\color{orange}$\HN(\gensubvb)$};
		\path (q0) ++(0.2, 2) node {\color{violet}$\HN(\genquotvb) = \HN(\modif{\genquotvb})$};
		\path (p0) ++(0.7, -0.5) node {\color{mynicegreen}$\HN(\modif{\genvb})$};
		\path (r1) ++(1.6, 0.9) node {\color{magenta}$\HN(\modif{\gensubvb})$};

\end{tikzpicture}
\caption{Illustration of Example \ref{example of minuscule modification without common breakpoints}}
\end{figure}
\end{example}

\bibliographystyle{amsalpha}

\bibliography{Bibliography}
	
\end{document}